\documentclass[11pt,letter]{amsart}
\usepackage{amsfonts,amssymb,amsthm,amsmath,amsxtra,amscd,verbatim,eucal}
\usepackage[all]{xy}
\usepackage[dvips]{graphics}

\theoremstyle{plain}
\newtheorem{thm}{Theorem}[section]
\newtheorem{lem}[thm]{Lemma}
\newtheorem{prop}[thm]{Proposition}
\newtheorem{cor}[thm]{Corollary}

\theoremstyle{definition}
\newtheorem{defi}[thm]{Definition}

\newtheorem{question}[thm]{Question}

\theoremstyle{remark}
\newtheorem{eg}[thm]{Example}

\newtheorem{rmk}[thm]{Remark}
\newtheorem{setting}[thm]{Setting}

\def\Z{{\mathbb Z}}

\def\F{{\mathbb F}}
\def\C{{\mathbb C}}
\def\A{{\mathbb A}}
\def\R{{\mathbb R}}
\def\Q{{\mathbb Q}}
\def\P{{\mathbb P}}

\def\E{\mathcal{E}}

\def\GG{\mathcal{G}}

\def\O{\mathcal{O}}
\def\OO{\mathcal{O}}
\def\LL{\mathcal{L}}
\def\M{\mathcal{M}}

\def\a{\alpha}

\def\g{\gamma}
\def\d{\delta}
\def\f{\phi}

\def\ep{\epsilon}

\def\l{\lambda}

\def\m{\mu}

\def\r{\rho}

\def\x{\xi}

\def\D{\Delta}
\def\G{\Gamma}

\def\S{\Sigma}

\def\.{\cdot}
\def\~{\widetilde}
\def\^{\widehat}

\def\o{\circ}
\def\ov{\overline}
\def\rat{\dashrightarrow}

\def\inj{\hookrightarrow}

\def\lrd{\lfloor}
\def\rrd{\rfloor}

\renewcommand{\and}{ \ \, \text{and} \ \, }

\newcommand{\fall}{ \quad \text{for all} \ \, }

\DeclareMathOperator{\codim} {codim}

\DeclareMathOperator{\im} {Im}

\DeclareMathOperator{\NE} {NE}
\DeclareMathOperator{\CNE} {\ov{\NE}}
\DeclareMathOperator{\Amp} {Amp}

\DeclareMathOperator{\Proj} {Proj}

\DeclareMathOperator{\PEff} {PEff}
\DeclareMathOperator{\Pic} {Pic}

\DeclareMathOperator{\Ex} {Ex}

\DeclareMathOperator{\ord} {ord}

\DeclareMathOperator{\Div} {Div}

\DeclareMathOperator{\Nef} {Nef}

\DeclareMathOperator{\Supp} {Supp}

\DeclareMathOperator{\Vol} {Vol}

\DeclareMathOperator{\Cl} {Cl}

\DeclareMathOperator{\SBs} {SBs}
\DeclareMathOperator{\Bs} {Bs}

\DeclareMathOperator{\GN}{\mathcal{GN}}
\DeclareMathOperator{\GNE}{\mathcal{GNE}}
\DeclareMathOperator{\GMov}{\mathcal{GM}ov}
\DeclareMathOperator{\GPEff}{\mathcal{GPE}ff}
\DeclareMathOperator{\GNef}{\mathcal{GN}ef}
\DeclareMathOperator{\Mov} {Mov}

\title{Deformations of canonical pairs and Fano varieties}

\author{Tommaso de Fernex}
\address{Department of Mathematics, University of Utah, 155 South 1400 East,
Salt Lake City, UT 48112-0090, USA}
\email{defernex@math.utah.edu}

\author{Christopher D. Hacon}
\address{Department of Mathematics, University of Utah, 155 South 1400 East,
Salt Lake City, UT 48112-0090, USA}
\email{hacon@math.utah.edu}

\thanks{The first author was partially supported by NSF
CAREER grant no: 0847059. The second author was partially
supported by NSF research grant no: 0757897
and an AMS Centennial fellowship.}


\begin{document}

\begin{abstract}
This paper is devoted to the study of various aspects of deformations of log pairs,
especially in connection to questions related to the invariance of
singularities and log plurigenera.
In particular, using recent results from the minimal
model program, we obtain an extension theorem for adjoint divisors
in the spirit of Siu and Kawamata
and more recent works of Hacon and M$^{\rm c}$Kernan.
Our main motivation however comes from the study of deformations of Fano varieties.
Our first application regards the behavior of Mori chamber decompositions
in families of Fano varieties: we prove that, in the case of mild singularities,
such decomposition is rigid under deformation when the dimension is small.
We then turn to analyze deformation properties of toric Fano varieties,
and prove that every simplicial toric Fano variety with at most
terminal singularities is rigid under deformations
(and in particular is not smoothable, if singular).
\end{abstract}

\maketitle

\section{Introduction}

In \cite{Siu98,Siu02}, Siu proved that
if $f \colon X \to T$ is a smooth 1-parameter
family of projective manifolds, then the plurigenera
of the fibers $X_t = f^{-1}(t)$ are constant functions of $t \in T$.
This result and its proof have been extremely influential in the field.
First proven in the analytic setting, Siu's result was later understood
in the algebraic language (at least in the case of general type)
by Kawamata \cite{Kaw99a,Kaw2}.
Kawamata also pointed out that using Siu's techniques one can show that
canonical singularities are preserved under
small deformations. An analogous result on terminal singularities
was obtained in \cite{Nak04}. These singularities
arise naturally in the context of the minimal model program
(but one should be aware that there are other classes of singularities
of the minimal model program that are not preserved under small deformations).

All these results are consequences of certain extension properties
of pluricanonical forms from a divisor to the ambient variety.
For example, if $X_0 = f^{-1}(0)$ is a fiber of a morphism $f$, then
one must study the surjectivity of the restriction maps
$$
H^0(X,\O_X(mK_X)) \to H^0(X_0,\O_{X_0}(mK_{X_0})).
$$
Extension theorems of this type have opened the door to recent
progress in higher dimensional geometry
(see \cite{HM06,Tak07,HM07,BCHM,HM08}).

From the point of view of the minimal model program it is natural to look at
{\it pairs} $(X,D)$, where $X$ is a variety and
$D = \sum a_i D_i$ is an effective $\Q$-divisor. These pairs arise naturally
in a geometrically meaningful way in a variety of instances,
and the various notions of singularities immediately extended
to analogous notions for pairs.

The first part of this paper
is devoted to the study of several properties related to deformations of pairs.
We start with some basic properties
on the deformation invariance of their singularities
that generalize to pairs the aforementioned results of Kawamata and Nakayama
(cf. Proposition~\ref{c-ext}),
and apply these results to study small deformations of
log Fano and weak log Fano varieties (cf. Proposition~\ref{thm:deform-log-fanos}).

We then address the extension problem for line bundles of the form
$\O_X(m(K_X + D))$. Thinking of $D$ as a sort of {\it boundary} of $X$, one can
consider the sections of these line bundles as a kind of log pluricanonical forms
(of course at this level of generality this is just for the purpose of
intuition).
Remarkable extension results for these kinds of sections
have been recently obtained in \cite{HM07,BCHM,HM08}
and applied towards the minimal model program.
One should be aware that these extension results are not straightforward
generalizations: the presence
of a boundary $D$ may in fact affect substantially
the extendability of the sections (cf. Remark~\ref{rmk:F_2-ext}),
which explains the appearance in the mentioned extension results of
certain conditions on stable base loci.

Using the techniques of the minimal model program and the results
established in \cite{BCHM}, we obtain the following extension theorem,
in which we replace a technical condition on the stable base locus
of $K_X + D$ (appearing in \cite{HM07,HM08} and other works)
with a certain positivity condition. The reader will find in Theorem~\ref{t-extf}
a statement including also a version of the result
where instead of the positivity conditions, the usual
condition on the stable base locus of $K_X + D$ is imposed.

\begin{thm}\label{thm:intro:extension}
With the above notation, assume that $X_0$ is $\Q$-factorial and $(X_0,D|_{X_0})$
is a Kawamata log terminal pair with canonical singularities.
Assume that either $D|_{X_0}$ or $K_{X_0} + D|_{X_0}$ is big.
Suppose furthermore that the restriction map
on N\'eron--Severi spaces $N^1(X/T) \to N^1(X_0)$ is surjective,
and that there is a number $a > -1$ such that $D|_{X_0} - aK_{X_0}$ is ample.

Let $L$ be an integral Weil $\Q$-Cartier divisor whose support of $L$ does not contain
$X_0$ and such that $L|_{X_0} \equiv  k(K_X+D)|_{X_0}$ for some rational number $k>1$.
Then, after possibly shrinking $T$ near $0$, the restriction map
$$
H^0(X,\OO _X(L))\to H^0(X_0,\OO _{X_0}(L|_{X_0}))
$$
is surjective.
\end{thm}

As we shall see, allowing to work with Weil divisors will be essential in
the application of this result to deformations of toric varieties.

Theorem~\ref{thm:intro:extension} is based upon another result, Theorem~\ref{p-1},
which essentially states that, under suitable conditions,
each step (a flip or a divisorial contraction) of
a relative minimal model program on a family $f \colon X \to T$
induces the same kind of step in the minimal model program of a fiber of $f$.

It turns out that there is a surprising connection between these results
and certain rigidity properties of Fano varieties under deformation.
According to Mori Theory, the Mori cone of effective curves of a Fano variety
encodes information on the geometry of the variety.
In fact, it follows by results in \cite{BCHM} that
a Fano variety with mild singularities is a {\it Mori dream space}
in the sense of \cite{HK}. In particular, the subcone of the
Ner\'on--Severi space of the variety generated by movable divisors
admits a finite polyhedral decomposition into {\it Mori chambers}.
One of these chambers is the nef cone.
This chamber decomposition retains information on the
birational geometry of the variety in terms of its small $\Q$-factorial
birational modifications (cf. \cite{HK}).

One would like to understand the behavior of Mori chamber decompositions
in families of Fano varieties with mild singularities. In the following, let
$$
f \colon X \to T
$$
be a flat projective deformation, parametrized
by a smooth pointed curve $T \ni 0$,
of a Fano variety $X_0 = f^{-1}(0)$ with $\Q$-factorial
terminal singularities.

When $f$ is a smooth family of Fano manifolds,
it follows from a result on nef values due to Wi\'sniewski \cite{Wis,Wis2}
that the nef cone is locally constant in the family.
Wi\'sniewski's result is of a topological
nature, relying on the Hard Lefschetz Theorem and
the fact that the family is topologically locally trivial.
On the other hand, it appears that one is forced to allow singularities in order
to study how the whole Mori chamber decomposition varies in the deformation.

We start by investigating to which extent Wi\'sniewski's theorem on the
deformation of nef values (cf.  \cite{Wis})
can be generalized to the context of canonical log pairs.
In this direction, we obtain an optimal result for families of pairs
of dimension at most three. By contrast,
there are interesting examples of birational
hyperk\"ahler manifolds of dimension four that force us to impose
additional conditions in order to ensure the constancy of nef values
in such general setting when the dimension is greater than three
(cf. Example~\ref{eg:Huybrechts}
and Remark~\ref{rmk:hyperkahler}).
We refer the reader to Theorem~\ref{t-nef} and Corollary~\ref{cor:nef-value}
for the precise statements of our results on nef values.

Before even addressing the question on the deformation of Mori chamber decompositions,
one needs to make sure that the Ner\'on--Severi spaces deform continuously.
Since the topology may vary in the family once singularities are allowed,
even the fact that the Picard number of the fibers is constant is
a priori not obvious. We deduce this fact from a property
established in \cite{KM92} (cf. Proposition~\ref{p-N_1}).

It was pointed out by Lazarsfeld
and Musta\c t\u a that the pseudo-effective cone of $X_0$
is locally constant under deformation.
In a similar fashion, we have the following result (cf. Theorem~\ref{prop:movable-cones}).

\begin{thm}
With the above notation, the cone of movable divisors of $X_0$,
which supports the Mori chamber decomposition, is locally constant.
\end{thm}

Regarding the behavior of the Mori chamber decomposition itself, we obtain
the following (cf. Theorem~\ref{thm:mori-chamber-dec};
see Definition~\ref{def:1-can} for the definition of {\it $1$-canonical}).

\begin{thm}\label{thm:intro:Mori-chambers}
With the above notation, suppose that either
\begin{enumerate}
\item
$\dim X_0 \le 3$, or
\item
$\dim X_0 = 4$ and $X_0$ is 1-canonical (e.g., $K_{X_0}$ is Cartier).
\end{enumerate}
Then, locally near $0$, the Mori chamber decomposition of $X_0$ is preserved unaltered
by the deformation.
\end{thm}

We remark that that the above theorems
fail in general if one just relaxes the condition on singularities
from terminal to canonical, or if $X_0$ is only assumed to be a weak Fano variety
or a log Fano variety (cf. Remark~\ref{eg-controexamples}).

We had originally conjectured in a earlier version of this paper that
the conclusions of Theorem~\ref{thm:intro:Mori-chambers} would hold in general,
without the restrictions on dimension and singularities imposed in~(a) or~(b).
Recently, however, Burt Totaro communicated to us that he has found
counter-examples to such conjecture (cf. \cite{Tot09}).
In particular, in view of Totaro's examples, the additional restrictions
considered in
Theorem~\ref{thm:intro:Mori-chambers} would appear to be optimal.

Our methods also show that, in the setting considered above,
the Mori chamber decomposition is locally invariant if the central fiber
is a toric variety. This is in fact just a hint of a much stronger
rigidity property (we remark, on the other hand, that the corresponding
assertions in Theorem~\ref{t-nef} and Corollary~\ref{cor:nef-value}
are not implied by such rigidity property).
More precisely, by analyzing the induced deformation on the total
coordinate ring of $X_0$, we prove
the following result (cf. Theorem~\ref{thm:rigidity-of-toric}).

\begin{thm}\label{thm:intro:rigidity-of-toric}
Every simplicial toric Fano variety with at most terminal singularities
is rigid under small projective flat deformations (and thus is not
smoothable, if singular).
\end{thm}

In the case of smooth toric Fano varieties, this theorem recovers
a special case of a result of Bien and Brion \cite{BB96}.
One should contrast this result with several known examples of
degenerations to (singular) toric Fano varieties,
notably of Grassmannians, flag varieties
and other moduli spaces (e.g., see
\cite{Str93,Str95,Lak95,GL96,BCFKvS00,Bat04,AB04,HMSV08}).
In particular, the degenerations studied in \cite{Bat04},
for instance, show how rigidity
fails if one drops the hypothesis on $\Q$-factoriality, and
there are simple examples of smoothable $\Q$-factorial toric Fano varieties
with canonical singularities and of non-rigid smooth toric varieties
that are weak Fano or log Fano (cf. Remark~\ref{rmk:F_2-toric}).
One should also bear in mind the fact that all
Fano 3-folds with Gorenstein terminal singularities (as well as all weak Fano
$\Q$-factorial 3-folds with Gorenstein terminal singularities) can be smoothed
(see \cite{Nam97,Min01}). This does not contradict the above
theorem, as in dimension three
all $\Q$-factorial terminal Gorenstein toric Fano varieties are already smooth.

\subsection*{Acknowledgments}
We would like to thank
V.~Alexeev,
S.~Boucksom,
M.~Brion,
M.~Hering,
A.~Kasprzyk,
J.~Koll\'ar,
R.~Lazarsfeld,
J.~M\textsuperscript{c}Kernan,
M.~Musta\c t\u a,
S.~Payne
and B.~Totaro
for useful correspondence and conversations.
In particular we are grateful to J\'anos Koll\'ar for bringing
\cite{KM92} to our attention, to Robert Lazarsfeld for bringing
\cite{Wis} to our attention, and to Burt Totaro for informing us of
the results of his preprint \cite{Tot09}.

\section{Notation and conventions}

Throughout this paper we work over the field of complex numbers $\mathbb C$.
A {\it divisor} on a normal variety will be understood to be a Weil divisor.
For a proper morphism of varieties $f \colon X \to S$, we denote
by $N^1(X/Z)$ the Ner\'on--Severi space (with real coefficients)
of $X$ over $Z$, and by $N_1(X/Z)$ its dual space.
Linear equivalence (resp., $\Q$-linear equivalence) between divisors is
denoted by $\sim$ (resp., $\sim_\Q$); numerical equivalence
between $\Q$-Cartier divisors is denoted by $\equiv$.

The {\it stable base locus} $\SBs(D)$ of a $\Q$-divisor $D$ on a normal variety $X$
is defined as the intersection $\bigcap_m \Bs(|mD|)$
of the base loci of the linear systems $|mD|$
taken over all $m > 0$ sufficiently divisible.
For a morphism $f \colon X \to Z$, we denote
$\SBs(D/Z) := \bigcap_H \SBs(D + f^*H)$, where the intersection is taken
over all ample $\Q$-divisors $H$ on $Z$.

A {\it pair} $(X,D)$ consists of a positive dimensional normal variety $X$ and a
$\Q$-divisor $D$ such that $K_X+D$ is $\Q$-Cartier.
Let $\mu \colon Y\to X$ be a proper birational morphism from a normal variety $Y$.
We say that $\m$ is a  {\it log resolution} of $(X,D)$
if the exceptional set $\Ex(\mu)$ is a divisor and the set
$\Ex(\mu) \cup \Supp(K_Y - \mu^*(K_X + D))$, where $\mu_*K_Y = K_X$,
is a simple normal crossings divisor.
Any divisor $E$ on $Y$ is said to be a divisor {\it over} $X$.
If $E$ is a prime divisor, then the {\it log discrepancy} of $E$ over $(X,D)$
is defined as $a_E(X,D) := 1+\ord_{E}(K_Y - \mu ^*(K_X+D))$, where the canonical
divisor $K_Y$ on $Y$ is chosen so that $\m_*K_Y = K_X$.
The {\it minimal log discrepancy} of $(X,D)$ is the infimum of
all log discrepancies of prime divisors over $X$.
Minimal log discrepancies can be either $-\infty$ of $\ge 0$, and in the second case
are computed as the minimum of the log discrepancies of the prime
divisors on any log resolution of $(X,D)$.
The pair $(X,D)$ is {\it Kawamata log terminal} (resp. {\it log canonical})
if $a_E(X,D)>0$ (resp. $a_E(X,D)\geq 0$) for any prime divisor $E$ over $X$;
this condition can be tested on the prime divisors on any log resolution of $(X,D)$.
The pair is {\it canonical} (resp. {\it terminal}) if
$a_E(X,D)\geq 1$ (resp. $a_E(X,D)>1$) for any
prime divisor $E$ exceptional over $X$, and moreover
$\lrd (1-\epsilon)D \rrd = 0$ for all small $\ep > 0$
(this last condition is only relevant when $\dim X = 1$,
as it is redundant otherwise; note that when $\dim X = 1$ the first
conditions are empty, since there are no exceptional divisors at all).

A variety $X$ is {\it Fano} if it is a positive dimensional
projective variety with ample $\Q$-Cartier anti-canonical divisor $-K_X$.
A pair $(X,D)$ is a {\it log Fano variety} (resp. a {\it weak log Fano variety})
if $X$ is projective, $(X,D)$ has Kawamata log terminal singularities, and
$-(K_X + D)$ is ample (resp. nef and big).

\section{Deformation properties of $\Q$-factoriality and
singularities of pairs}\label{sect:def-canonical-pairs}

We start by recalling the following result of Koll\'ar and Mori.

\begin{prop}\label{prop:KM}
Let $X$ be a normal variety and $S\subset X$ a Cartier divisor that
is normal and satisfies Serre's condition $S_3$
(cf. \cite[Definition~5.2]{KM}).
Let $D$ be any Weil divisor on $X$ whose support does not contain $S$.
If $D|_S$ (defined as the restriction of $D$ as a Weil divisor)
is Cartier, and there is a closed subset $Z\subset X$
with ${\rm codim}(Z\cap S,S)\geq 3$ such that
$D$ is Cartier on $U=X-Z$, then
$\O _X(D)$ is Cartier on a neighborhood of $S$.
\end{prop}

For the convenience of the reader, we include the arguments of \cite[12.1.8]{KM92}.

\begin{proof}
Since the question is local, we may assume that $\O_X(S) \cong \O_X$,
and hence $\O_U(S|_U) \cong \O_U$.
Let $i\colon U\to X$ be the inclusion.
Multiplication by the restriction of a section of $\O_X(S)$ defining $S$
gives a short exact sequence on $U$
$$
0\to \O _U(D|_U)\cong \O_U((D-S)|_U)
\to \O_U(D|_U)\to \O _{U\cap S}(D|_{U\cap S})\to 0.
$$
Pushing forward via $i$ and using the fact that as $S$ is $S_3$, then
$R^1i_*\O _{U\cap S}(D|_{U\cap S})=H^2_{Z\cap S}(\O _{S}(D|_{S}))=0$, we
obtain an exact sequence
$$
0\to \O _X(D)\to \O_X(D)\to \O _{S}(D|_{ S})
\to R^1i_*\O _U(D|_U)\to R^1i_*\O_U(D|_U)\to 0.
$$
The map $R^1i_*\O _U(D|_U)\to R^1i_*\O_U(D|_U)$ is induced by multiplication
by the equation of $S$. By Nakayama's Lemma, it follows that $R^1i_*\O_U(D|_U)=0$
on a neighborhood of $S$.
But then $\O_X(D)\to \O _{S}(D|_{ S})$ is surjective and the claim follows.
\end{proof}

\begin{cor}\label{cor:factoriality}
Let $S$ be a normal irreducible Cartier divisor
on a variety $X$, and suppose that $S$
is $\Q$-factorial with terminal singularities.

Then $X$ is $\Q$-factorial in a neighborhood of $S$.
\end{cor}

\begin{rmk}
The fact that $X$ is normal in a neighborhood of $S$ follows by
\cite[Corollary~5.12.7]{Gro}.
\end{rmk}

\begin{rmk}
The corollary fails in general if the singularities are canonical but not terminal.
For an example, one can consider the family of quadrics of equation
$\{xy + z^2 + t^2u^2 = 0\} \subset \P^3 \times \A^1$, where $(x,y,z,u)$
are homogeneous coordinates on $\P^3$ and $t$ is a parameter on $\A^1$.
\end{rmk}

\begin{prop}\label{c-ext}
On a normal variety $X$, let $S$ be
a normal Cartier divisor, and $D$ be a
divisor whose support does not contain $S$.
Assume that $K_S + D|_S$ is $\Q$-Cartier and $(S,D|_S)$ is a
canonical pair.

Then $K_X+S+D$ is $\Q$-Cartier on a neighborhood of $S$. Moreover:
\begin{enumerate}
\item
If $\lrd D|_{S} \rrd = 0$,
then $(X,S+D)$ is purely log terminal
and $(X,D)$ is canonical in a neighborhood of $S$.
In particular, if $f\colon X\to T$ is a flat morphism and $S=X_0$
is the fiber over a point $0 \in T$, then
$(X_t,D|_{X_t})$ is canonical for every $t$
in a neighborhood of $0$.
\item
If $(S,D|_S)$ is terminal and $\lrd D|_{S} \rrd = 0$,
then $(X,D)$ is terminal in a neighborhood of $S$.
In particular, if $f\colon X\to T$ is a flat morphism and $S=X_0$
for some $0 \in T$, then $(X_t,D|_{X_t})$ is terminal for every $t$
in a neighborhood of $0$.
\item
If $S$ is $\Q$-Gorenstein (equivalently, if $D|_{S}$ is $\Q$-Cartier),
then $(X,S+D)$ is log canonical in a neighborhood of $S$.
\end{enumerate}
\end{prop}

\begin{proof} Since $(S,D|_S)$ is canonical, $S$ is Cohen--Macaulay.
The case in which $\dim S = 1$ is trivial and the one in which
$\dim S=2$ is well known. We therefore assume that
$\dim S\geq 3$ and hence $S$ is $S_3$.
By the $2$-dimensional case, it follows that $(X,D)$ is canonical
in codimension $2$ near $S$. Therefore the closed subset $Z\subset X$
on which $K_X+S+D$ is not $\Q$-Cartier has codimension at least $3$.
By Proposition~\ref{prop:KM}, $K_X+S+D$ is $\Q$-Cartier on a neighborhood of $X$.

If $\lrd D|_{S} \rrd = 0$, then $(S,D|_{S})$ is
Kawamata log terminal, and since
we have now established that $K_X+S+D$ is $\Q$-Cartier,
it is well known that $(X,S+D)$ is purely log terminal
in a neighborhood of $S$ (cf. \cite[Theorem 5.50]{KM}).
By the arguments of \cite[Corollary 1.4.3]{BCHM}
it follows that $(X,D)$ is canonical in a neighborhood of $S$.

Assume now that $(S,D|_S)$ is terminal and $\lrd D|_{S} \rrd = 0$,
and suppose that $(X,D)$ is not terminal near $S$.
Then there is an exceptional divisor over $X$ whose
discrepancy over $(X,D)$ is $\le 0$ and whose center
$W \subset X$ intersects $X_0$. We can find an effective Cartier divisor
$H \subset X$ containing $W$ and intersecting
properly $X_0$. Note that the pair $(X,D + \ep H)$ is not canonical at $W$
for any $\ep > 0$. On the other hand, if $\ep$ is sufficiently
small, then $(X_0,D_{X_0} + \ep H_{X_0})$
is canonical and $\lrd D_{X_0} + \ep H_{X_0}\rrd = 0$.
This contradicts~(a).

Regarding the assertion in~(c), if $S$ is $\Q$-Gorenstein,
then $K_{S} + (1-\ep)D|_{S}$ is $\Q$-Cartier and $(S,(1-\ep)D|_{S})$ is
Kawamata log terminal for every rational number $\ep > 0$,
so that we conclude that
$(X,S+(1-\ep)D)$ is purely log terminal in a neighborhood of $S$ and hence
$(X,S+D)$ is log canonical on this neighborhood.
\end{proof}

\begin{rmk}
The example \cite[Example~4.3]{Kaw2} shows that this result fails for
Kawamata log terminal singularities.
\end{rmk}

\begin{rmk}
For a pair $(X,D)$ with canonical singularities, the property that
$\lrd D \rrd \le 0$ is equivalent to the pair being
Kawamata log terminal.
\end{rmk}

Since ampleness is open in families, it follows in particular
by Proposition~\ref{c-ext} that every small flat deformation
of a log Fano variety with canonical/terminal singularities
is a log Fano variety with canonical/terminal singularities.
In fact, the same holds for weak log Fano varieties as well.

\begin{prop}\label{thm:deform-log-fanos}
Let $f\colon X\to T$ be a flat projective fibration from a normal
variety to a smooth curve. Fix a point $0 \in T$,
and suppose that the fiber $X_0$ over $0$ is a normal variety.
Assume that for some effective
divisor $D$ on $X$ not containing $X_0$ in its support,
the pair $(X_0,D|_{X_0})$ is a log Fano variety (resp. a weak log Fano variety)
with canonical/terminal singularities.

Then $(X_t,D|_{X_t})$ is a log Fano variety (resp. a weak log Fano variety)
with canonical/terminal singularities for all $t$ in a neighborhood of $0\in T$.
\end{prop}

\begin{proof} As $(X_0,D|_{X_0})$ is a log Fano variety
(resp. a weak log Fano variety), we have that $\lfloor D|_{X_0}\rfloor =0$.
Note that, since $X_0$ is a connected and reduced fiber,
$f$ has connected fibers. Since $(X_0,D|_{X_0})$ is canonical,
the $\Q$-divisor $K_X + D$ is $\Q$-Cartier in
a neighborhood of $X_0$ by Proposition~\ref{c-ext}, and the pair
$(X_t,D|_{X_t})$ is canonical/terminal for all $t$ in a neighborhood of $0$.
Notice also that $\lrd D \rrd = 0$ in a neighborhood of $X_0$,
and hence $\lrd D|_{X_t} \rrd = 0$ for $t$ in a neighborhood of $0$.
If $(X_0,D|_{X_0})$ is log Fano, then $-(K_{X_t}+D|_{X_t})$ is ample
for every $t$ in a neighborhood of $0\in T$ since ampleness is
open in families, and hence $(X_t,D|_{X_t})$ is log Fano for any such $t$.
Assuming that $(X_0,D|_{X_0})$ is only weak log Fano,
we conclude by the lemma that follows.
\end{proof}

\begin{lem}\label{lem:nef-big-open}
Let $(X,D)$ be a Kawamata log terminal pair (with $D$
an effective $\Q$-divisor). Suppose that
$f\colon X\to T$ is a flat projective fibration
to a smooth curve, and let $0 \in T$ be a point such that the fiber
$X_0$ is a normal variety not contained in the support of $D$.
Let $L$ be a $\Q$-Cartier $\Q$-divisor on $X$ such that $L|_{X_0}$ is nef and big
and so is $aL|_{X_0}-(K_X + D)|_{X_0}$ for some $a \ge 0$.
Then $L|_{X_t}$ is nef and big
for all $t$ in a neighborhood of $0\in T$.
\end{lem}

\begin{proof}
By re-scaling, we may well assume that $L$ is a Cartier divisor.
Let $M \in \{L, raL -r(K_X + D)\}$, where $r$ is the Cartier index
of $aL -(K_X + D)$ (we can assume that $a$ is rational).
Note that $M|_{X_0}$ is nef and big.
Then $M|_{X_t}$ is nef if $t \in T$ is very general
(see \cite[Proposition~1.4.14]{Laz}).

We claim that $M|_{X_t}$ is also big for very general $t \in T$.
Indeed, since $M|_{X_0}$ is nef and big, all the higher cohomology
groups $H^i(kM|_{X_0})$
fail to grow maximally as $k \to \infty$; more precisely, there is
a constant $C$ such that
$h^i(\O_{X_0}(kM|_{X_0})) \le Ck^{d-1}$ for every $i > 0$ and $k \ge 1$
where $d$ is the dimension of $X_0$.
Note that $f$ is equidimensional.
By semicontinuity of cohomological dimensions \cite[III.12.8]{Har},
for every $k$ there is a open neighborhood $U_k \subseteq T$ of $0$
such that
$h^i(\O_{X_t}(kM|_{X_t})) \le Ck^{d-1}$ for every $i > 0$ and $t \in U_k$.
We conclude that $h^i(\O_{X_t}(kM|_{X_t}))$ fails to grow maximally
for very general $t$. On the other hand, note that $\O_X(kM)$ is
flat over $T$ (cf. \cite[III.9.2(e) and~III.9.2(c)]{Har}),
and thus the Euler characteristic
of $\O_{X_t}(kM|_{X_t})$ is constant as a function of $t$.
We conclude that $h^0(\O_{X_t}(kM|_{X_t}))$ grows maximally
(i.e., to the order of $k^d$) for very general $t$.
This shows that $M|_{X_t}$ is big for very general $t \in T$.

So there is a set $W \subseteq T$, which is the complement
of a countable set, such that
$(X_t,D|_{X_t})$ is Kawamata log terminal and both
$L|_{X_t}$ and $aL|_{X_t}-(K_X + D)|_{X_t}$ are nef and big for every $t \in W$.
By \cite[Theorem~1.1]{Kollar},
there is a positive integer $m$ (depending only on $d$ and $a$)
such that $|mL|_{X_t}|$ is base point free for every $t \in W$,
and for any such $t$ the associated morphism $\f_{|mL|_{X_t}|}$ is generically finite
since $L|_{X_t}$ is big. By further shrinking $W$ if necessary,
we can also assume that the restriction map
$H^0(\O_X(mL)) \to H^0(\O_{X_t}(mL))$
is surjective (cf. \cite[III.12.8 and~III.12.9]{Har}).
This implies that $\f_{|mL|_{X_t}|} = \f_{|mL|}|_{X_t}$ for all $t \in W$,
and hence we conclude that $\f_{|mL|}|_{X_t}$ is generically finite
for a general $t \in T$. This implies that $L|_{X_t}$ is big
for every $t$ in a neighborhood of $0$.

Consider now the subset
$\S = \{ t \in T \mid \text{$L|_{X_t}$ is not nef} \}$.
We consider the base locus ${\rm Bs}(|mL|)$ of $|mL|$.
Note that this locus is a Zariski closed subset of $X$.
We observe that ${\rm Bs}(|mL|) \cap X_t \ne \emptyset$
for all $t \in \S$, since it clearly contains the locus where
$L$ is not nef. On the other hand,
taking $t$ very general so that $|mL|_{X_t}|$ is base point free,
we have ${\rm Bs}(|mL|_{X_t}|) = \emptyset$,
and since we can also ensure that the restriction map
$H^0(\O_X(mL)) \to H^0(\O_{X_t}(mL))$
is surjective, we conclude that
${\rm Bs}(|mL|) \cap X_t = \emptyset$.
This implies that $\S$ is a finite set, and hence
$L|_{X_t}$ is also nef for every $t$ in a neighborhood of $0$.
\end{proof}

The above lemma will also turn out to be useful in the next section.

\section{Invariance of plurigenera for canonical pairs}\label{sect:exteS:extnsion}

This section is the technical core of the paper.
Our first result relates to the application of the minimal
model program on a family of varieties
(see \cite[12.3]{KM92} for a related statement).

\begin{thm}\label{p-1}
Let $f\colon X\to T$ be a flat projective morphism of normal varieties
where $T$ is an affine curve. Assume that for some
$0 \in T$ the fiber $X_0$ is normal
and $\Q$-factorial, $N^1(X/T)\to N^1(X_0)$ is
surjective, and $D$ is an effective $\Q$-divisor whose support
does not contain $X_0$ such that $(X_0,D|_{X_0})$ is a Kawamata log terminal pair
with canonical singularities. Let $\psi \colon X\to Z$ be the contraction over $T$
of a $(K_X +D)$-negative extremal ray of $\CNE(X/T)$, and let $Z_0 = \psi(X_0)$.

If $\psi_0 := \psi|_{X_0} \colon X_0 \to Z_0$ is not an isomorphism,
then it is the contraction of a
$(K_{X_0} +D|_{X_0})$-negative extremal ray, and:
\begin{enumerate}
\item
If $\psi$ is of fiber type, then so is $\psi_0$.
\item
If $\psi$ is a divisorial contraction of a divisor $G$,
then $\psi _0$ is a divisorial
contraction of $G|_{X_0}$ (in particular $G|_{X_0}$ is irreducible),
and $N^1(Z/T)\to N^1(Z_0)$ is surjective.
\item
Assume additionally that either
\begin{itemize}
\item
${\rm SBs}(K_X+D/Z)$ does not contain any
component of ${\rm Supp}(D|_{X_0})$, or
\item
$D|_{X_0}-aK_{X_0}$ is nef over $Z_0$ for some $a>-1$.
\end{itemize}
If $\psi$ is a flipping contraction and
$\psi ^+\colon X^+\to Z$ is the flip, then $\psi _0$ is a
flipping contraction and, denoting $X_0^+$ the proper transform of $X_0$
on $X^+$, the induced morphism $\psi ^+_0\colon X^+_0\to Z_0$ is the flip
of $\psi _0\colon X_0\to Z_0$. Moreover
$N^1(X^+/T)\to N^1(X_0^+)$ is surjective.
\end{enumerate}
\end{thm}

\begin{proof}
We will repeatedly use the fact that $T$ is affine so that
for example a divisor is ample over $T$ if and only if it is ample.
For short, for a divisor $L$ on $X$ not containing $X_0$ in its
support, we denote $L_0 := L|_{X_0}$.

If $\psi$ is of fiber type, then it follows by the semicontinuity
of fiber dimension (applied to $\psi$) that $\psi_0$ is also of fiber type.
We can henceforth assume that $\psi$ is birational.
Note that this implies that $\psi_0$ is birational as well,
by the semicontinuity of fiber dimension applied this time
to the morphism $Z \to T$.

Note that as $N^1(X/T)\to N^1(X_0)$ is surjective,
any curve in $X_0$ that spans a $(K_X+D)$-negative extremal
ray of $\CNE(X/T)$ also spans a
$(K_{X_0}+D_0)$-negative extremal ray of $\CNE(X_0)$.
There is an ample $\Q$-divisor $H$ on $X$ such that
the contraction $\psi$ is defined
by $|m(K_X+D+H)|$ for any sufficiently divisible $m \ge 1$.
Since $K_X+D+H$ is nef and big and $X_0 \sim 0$, the restriction map
$$
H^0\big(\O_X(m(K_X+D+H))\big) \to H^0\big(\O_X(m(K_{X_0}+D_0+H_0))\big)
$$
is surjective if $m$ is sufficiently divisible by
Kawamata--Viehweg's vanishing theorem (cf. \cite[2.70]{KM}). This implies that
$\psi_0$ coincides with the map defined by $|m(K_{X_0}+D_0+H_0)|$
for all sufficiently divisible $m$,
and thus it is the extremal contraction of the ray in question.
In particular, $\psi_0$ has connected fibers and $Z_0$ is normal.

Assume that $\psi$ is a divisorial contraction and
let $G$ be the (irreducible) divisor contracted by $\psi$.
Then $G$ dominates $T$ and as $\psi (G)$ is irreducible of
dimension $\leq \dim Z-2$, all components of $G_0$ are contracted.
It follows from the injectivity of $N_1(X_0) \to N_1(X/T)$
that $\psi _0\colon X_0\to Z_0$ is an
extremal (divisorial) contraction of $G_0$, and
in particular $G_0$ is irreducible.

Assume that $\psi$ is a flipping contraction, and
suppose by way of contradiction that
$\psi _0$ is a divisorial contraction, so that $Z_0$ is $\Q$-factorial.
We consider two cases according to which hypothesis we choose.
If we are assuming that ${\rm SBs}(K_X+D/Z)$ does not contain any
component of $D_0$, then $(Z_0,(\psi _0)_*D_0)$ is canonical,
and thus $K_Z + \psi_*D$ is $\Q$-Cartier by Proposition~\ref{c-ext},
which is impossible since $\psi$ is a $(K_X + D)$-negative contraction.
Similarly, if we are assuming
that $D|_{X_0}-aK_{X_0}$ is nef over $Z_0$, then
$$
-(a+1)K_{X_0}=-(K_{X_0}+D|_{X_0})+(D|_{X_0}-aK_{X_0})
$$
is ample over $Z_0$. Since $a+1 > 0$, this implies that $Z_0$ is canonical, and hence
$K_Z$ is $\Q$-Cartier (again by Proposition~\ref{c-ext}),
which is impossible since $\psi$ is now $K_X$-negative.
Therefore $\psi _0$ is a flipping contraction.

Let
$$
\f\colon X\dasharrow X^+
:=\Proj_{Z}\Big(\bigoplus _{m\geq 0}\psi _*\O _X(m(K_{X}+D))\Big)
$$
be the flip and $D^+=f _*D$,
and let $p \colon Y \to X$, $q \colon Y \to X^+$ be a common resolution of $\f$.
Let $X_0^+$ and $Y_0$ be the strict transform of $X_0$ in, respectively, $X^+$ and $Y$.
Note that $X_0^+$ is normal (as $(X^+,D^+ + X_0^+)$ is purely log terminal), and
we can assume that the induced maps $p_0 \colon Y_0 \to X_0$, $q_0 \colon Y_0 \to X^+_0$
give a common resolution of the restriction $\f_0$ of $\f$ to $X_0$.
Since $\f$ is a $(X,D)$-flip, we have
$$
p^*(K_X+D) \sim_\Q q^*(K_{X^+} + D^+) + F,
$$
where $F$ is effective and dominates both $\Ex(\psi)$ and $\Ex(\psi^+)$.
Restricting to $Y_0$ and
using adjunction (note that the divisors $X_0$ and $X_0^+$ are linearly equivalent to
zero, so we can add them in the above relation), we obtain
$$
p_0^*(K_{X_0}+D|_{X_0}) \sim_\Q q_0^*(K_{X_0^+} + D^+|_{X_0^+}) + F|_{Y_0}.
$$
As $(X_0,D|_{X_0})$ is canonical, it follows that the inverse map
$\f_0^{-1} \colon X^+_0\dasharrow X_0$
contracts no divisors. Since $X_0\to Z_0$ also contracts no divisors,
it follows that $X_0^+ \to Z_0$ is a small contraction,
and thus $\f_0\colon X_0\dasharrow X^+_0$ is an isomorphism in codimension one
and $\phi \colon X \rat X^+$ is an isomorphism at every codimension one point of $X_0$.
In particular, $D^+|_{X^+_0}=\f_{0*}D|_{X_0}$. But then one sees that
$X^+_0$ is the log canonical model
of $(X_0,D|_{X_0})$ over $Z_0$, and hence $X^+_0\to Z_0$ is the flip of $X_0\to Z_0$.
The surjectivity $N^1(X^+/T)\to N^1(X^+_0)$ follows easily.
\end{proof}

\begin{rmk}
Note that, as we have seen in the proof, the condition that
$D|_{X_0} - a K_{X_0}$ is nef over $Z_0$ for some $a > -1$ considered in case~(c)
is equivalent to the condition that $-K_{X_0}$ is ample over $Z_0$.
\end{rmk}

\begin{rmk}
From the discussion in the next section, it follows that if $f\colon X \to T$
is a family of terminal $\Q$-factorial
Fano varieties, then one can always reduce by finite base change to
a setting where $N^1(X/T) \to N^1(X_0)$ is surjective.
Note, moreover, that for a family of Fano varieties, $D-aK_X$ is ample over $T$ for
every $a \gg 0$.
\end{rmk}

\begin{rmk}\label{rmk:F_2}
The necessity of the additional hypotheses in case~(c) is manifested
in the following example. Let $X \to T = \A^1$ be a flat family of
smooth 2-dimensional quadrics degenerating to
$\F_2 := \P(\O_{\P^1}\oplus\O_{\P^1}(-2))$, and let $E$ be the $(-2)$-curve
on the central fiber $X_0 = \F_2$ (note that $K_X\.E=0$).
A line in one of the two rulings on the general
fiber $X_t$ sweeps out, under deformation
over $T$, a divisor $R$ on $X$ such that $R|_{X_0} = E + F$,
where $F$ is a fiber of the projection $\F^2 \to \P^1$.
If $0 < \ep \ll 1$, then
$(X,\ep R)$ is a (log smooth) log Fano variety with terminal singularities.
The small contraction $X \to Z$ of $E$ is a
$(K_X + \ep R)$-negative extremal contraction, and restricts to a divisorial
contraction of $X_0$.
Therefore the conclusions of Theorem~\ref{p-1} fail in this example.
\end{rmk}

As an application, we prove the following extension theorem.

\begin{thm}\label{t-extf}
Let $f\colon X\to T$ be a flat projective morphism of normal varieties
where $T$ is an affine curve. Assume that for some
$0 \in T$ the fiber $X_0$ is normal
and $\Q$-factorial, $N^1(X/T)\to N^1(X_0)$ is
surjective, and $D$ is an effective $\Q$-divisor whose support
does not contain $X_0$ such that $(X_0,D|_{X_0})$ is a Kawamata log terminal pair
with canonical singularities. Assume that either $D|_{X_0}$ or $K_{X_0}+D|_{X_0}$
is big, and that one of the following two conditions is satisfied:
\begin{enumerate}
\item
${\rm SBs}(K_X+D)$ does not contain any component of ${\rm Supp}(D|_{X_0})$, or
\item
$D|_{X_0}-aK_{X_0}$ is ample for some $a>-1$.
\end{enumerate}
Let $L$ be an integral Weil $\Q$-Cartier divisor whose support of $L$ does not contain
$X_0$ and such that $L|_{X_0} \equiv  k(K_X+D)|_{X_0}$ for some rational number $k>1$.
Then the restriction map
$$
H^0(X,\OO _X(L))\to H^0(X_0,\OO _{X_0}(L|_{X_0}))
$$
is surjective.
\end{thm}

\begin{proof}
It follows by Proposition~\ref{c-ext} that $K_X$ and $K_X+D$ are $\Q$-Cartier
and $(X,D)$ is canonical in a neighborhood of $X_0$.
After possibly shrinking $T$ around $0$, we can assume that
$(X,D)$ is everywhere canonical and both $K_X$ and $D$ are $\Q$-Cartier.
In case~(b), we can also assume that $D - aK_X$ is ample.
Notice that this reduction does not affect the assumption
that $N^1(X/T)\to N^1(X_0)$ be surjective.
Note also that we can assume that
$K_{X_0}+D|_{X_0}$ is pseudo-effective, as otherwise there is nothing to prove.

We run a $(X,D)$-minimal model program over $T$
$$
X = X^0 \rat X^1 \rat X^2 \rat \dots
$$
For a divisor $A$ on $X$, we denote by $A^{i}$ its proper transform on $X^{i}$.
If we are in case~(b), then we run the minimal model program with scaling of $D-aK_{X}$,
whereas in case~(a) we proceed with scaling of some fixed ample divisor.

Using Theorem \ref{p-1}, one sees inductively that if we are in case~(a),
then for every $i$ the stable base locus ${\rm SBs}(K_{X^{i}}+D^{i})$ does not contain
any component of ${\rm Supp}(D^{i}|_{X_0^{i}})$. In case~(b),
the Mori contraction $\psi_i\colon X_i \to Z_i$ corresponding to the step
$X_i \rat X_{i+1}$ is $(K_{X^i} + D^i)$-negative and
$((K_{X^i} + D^i) + \l_i (D^i - aK_{X^i}))$-trivial for some $\l_i >0$, and thus
is $K_{X^i}$-negative since
$$
\l_i(1+a)K_{X^{i}}=(1+\l_i)(K_{X^{i}}+D^{i})- ((K_{X^i} + D^i) + \l_i(D^{i}-aK_{X^{i}})).
$$
In either case,
it follows by Theorem \ref{p-1} that each divisorial (resp., flipping) contraction on $X$
corresponds to a divisorial (resp., flipping) contraction on $X_0$, and the
corresponding divisorial or flipping contraction on $X$
induces a divisorial or flipping contraction on $X_0$.
Therefore the given $(X,D)$-minimal model program over $T$ induces a
$(X_0,D|_{X_0})$-minimal model program.

Since in case~(a) we are assuming that $K_{X_0}+D|_{X_0}$ is big,
and in case~(b) we have that $K_{X_0}+D|_{X_0}$ is pseudo-effective and $D|_{X_0}$ is big,
it follows by \cite{BCHM} that, in either case,
the $(X_0,D|_{X_0})$-minimal model program terminates,
which implies that the $(X,D)$-minimal model program also terminates.
Therefore, we have a $(X,D)$-minimal model $\psi \colon X \dasharrow X'$
which induces a minimal model $\psi _0 \colon X_0 \dasharrow X'_0$ for
$(X_0,D|_{X_0})$. Note that $\psi _*(K_X+D)|_{X'_0}=\psi _{0*}(K_{X_0}+D|_{X_0})$.

Let $p\colon Y\to X$ and $q\colon Y\to X'$ be a common resolution
of $\psi$, let $Y_0$
be the strict transform of $X_0$, and let $p_0=p|_{Y_0}$ and $q_0=q|_{Y_0}$.
Let $L'=\psi _*L$ and $L'_0=\psi _{0*}(L_0)$,
where $L_0 = L|_{X_0}$, so that $L'_0=L'|_{X'_0}$.
We then have that $p^*L=q^*L'+E$ and $p_0^*L_0=q^*L'_0+E_0$
where $E\geq 0 $ is $q$-exceptional and $E_0\geq 0$ is $q_0$-exceptional.
This implies that
$$
H^0(X,\O _X(L))\cong H^0(Y,\O _Y(\lceil p^* L \rceil ))=
H^0(Y,\O _Y(\lceil q^* L' +E\rceil ))\cong H^0(X',\O _{X'}(L'))
$$
and similarly that
$$
H^0(X_0,\O _{X_0}(L_0))\cong H^0(X'_0,\O _{X'_0}(L'_0)).
$$
(Here we have repeatedly used the fact that if $g\colon X\to W$ is a proper birational
morphism of normal varieties, $D$ is an integral $\Q$-Cartier divisor on $W$, and $F$
is an effective and $g$-exceptional $\Q$-divisor on $X$,
then $g_* \O _X(g^* D+F)=\O _Y(D )$.)

We have reduced ourselves to showing that
$H^0(X', \O_{X'}(L'))\to H^0(X'_0, \O_{X'_0}(L'_0))$ is surjective.
We will conclude this using the following two lemmas.
For short, let $D' = \phi _* D$.

\begin{lem}
There is a short exact sequence
$$
0\to \OO _{X'}(L'-X'_0)\to \OO _{X'}(L')\to \OO _{X'_0}(L'_0)\to 0.
$$
\end{lem}

\begin{proof}
By \cite[Proposition~5.26]{KM} and its proof, it suffices to show that
if $U \subseteq X'$ is the open subset where $L'$ is Cartier, then
$X_0' \cap (X \setminus U)$ has codimension at least 2 in $X_0'$.
This is clear, since $X_0$ is a normal Cartier divisor of $X'$,
and thus $X'$ is smooth at every codimension one point of $X_0'$.
\end{proof}

\begin{lem}
After possibly further shrinking $T$, we have $H^1(X', \OO _{X'}(L'-X'_0))=0$.
\end{lem}

\begin{proof}
Suppose first that we are in the case where $K_{X'}+D'$ is big. Note that it is also nef.
We observe that
$$
(L'-(K_{X'}+D'))|_{X_0'} \equiv (k-1)(K_{X'}+ D')|_{X'_0}
$$
is nef and big, and similarly
$$
\big(a(L'-(K_{X'}+D')) - (K_{X'}+D')\big)|_{X_0'} \equiv (a(k-1)-1)(K_{X'}+ D')|_{X'_0}
$$
is nef and big, provided $a > 1/(k-1)$.
Therefore, by Lemma~\ref{lem:nef-big-open}, the $\Q$-Cartier $\Q$-divisor
$L'-(K_{X'}+D')$ is nef and big in a neighborhood of $X_0$, and thus the
vanishing follows by \cite[Theorem~2.16]{Kol95}. In the case where
$D'$ is big over $T$, we may write $D'=G+H$ where $H$ is a suitable ample
$\Q$-divisor such that $(X',G)$ is Kawamata log terminal, and a similar argument
applied to this pair in place of $(X',D')$ gives the vanishing.
\end{proof}

We deduce from the two lemmas
that $H^0(X', \O_{X'}(L'))\to H^0(X'_0, \O_{X'_0}(L'_0))$ is surjective,
which completes the proof of the theorem.
\end{proof}

\begin{rmk}
We have the following related result due to Takayama (cf. \cite[4.1]{Tak07}):
If $S$ is a smooth divisor on a smooth variety $X$ and $L$ is a Cartier
divisor on $X$ such that $L\sim _{\Q}A+D$ where $A$ is ample,
$S$ is not contained in the support of $D$ and
$(S,D|_S)$ is Kawamata log terminal, then the natural map
$$
H^0(X,\O _X(m(K_X+S+L)))\to H^0(S,\O _S(m(K_S+L|_S)))
$$
is surjective for every $m>0$.
\end{rmk}

\begin{rmk}\label{rmk:F_2-ext}
The fact that in the example discussed in Remark~\ref{rmk:F_2} the
curve $E \subset X_0$ does not deform away from the central fiber
shows how the conclusion stated in
Theorem~\ref{t-extf} does not hold in general if one
does not impose any condition either on the stable base locus of $K_X + D$
or on the positivity of $D|_{X_0} - aK_{X_0}$ for some $a > -1$.
In fact, using the same example, one can also find counterexamples
to the extension property for divisors $K_X + D$ that are big over $T$.
In the following, we use the notations introduced in Remark~\ref{rmk:F_2}.
For any $b \ge 1$, consider a divisor $L$ on $X$ such that
$L|_{X_0} \sim (b+1)E + 2bF$. On a general fiber $X_t$ we have
$$
h^0(\O_{X_t}(L)) = h^0(\O_{\P^1 \times \P^1}(b+1,b-1)) = b^2 + 2b.
$$
On the other hand, on the central fiber we have
$$
|L|_{X_0}| = |b(E+2F)+E| = |b(E+2F)| + E,
$$
and since $|E+2F|$ is the linear system defining the morphism
$X_0 \to Z_0\subset \P^3$, we get
$$
h^0(\O_{X_0}(L)) = h^0(\O_{Z_0}(b)) = \binom{b+3}3 - \binom{b+1}3 = b^2 + 2b + 1.
$$
Therefore there are sections of $\O_{X_0}(L)$ that do not extend to $X$.
Note that $L$ is relatively big over $T$. Moreover, consider, for some $m \gg 1$,
a general $H \in |-\frac{b-1+2m}2 K_X|$ (note that $-K_X$ is semi-ample
and is divisible by 2 in $\Pic(X)$), and let $D := \frac 1m (H + 2R)$
(recall $R$ from Remark~\ref{rmk:F_2}). Then $(X,D)$ is a Kawamata log
terminal pair with terminal singularities and $L \sim m(K_X + D)$.
The reason why the mentioned theorems do not apply in this setting
is that the stable base locus of $L$ contains the curve $E$,
which is a component of the support of $D|_{X_0}$, and
there are no values of $a > -1$ (in fact, of $a \in \R$) such that
$D|_{X_0} - aK_{X_0}$ is ample.
\end{rmk}

\section{Nef values in families}\label{s-nef}

Given $\mathbb Q$-Cartier divisors $A$ and $B$ on a
normal projective variety $X$, with $A$ ample,
the {\it nef value} (or {\it nef threshold}) of $B$ with respect to $A$
is defined by
$$
\tau _A(B):= \min \{\lambda \in \R \mid \text{$B+\lambda A$ is nef}\}.
$$
This invariant is of particular interest in the adjoint case, namely when $B=K_X+D$
for some Kawamata log terminal pair $(X,D)$ such that $K_X+D$ is not nef.
In this case, it follows from the Cone Theorem that
$\tau=\tau _A(K_X+D) $ is rational and $K_X+D+\tau A$ is semiample.

We are interested in determining conditions that guarantee
that the nef values of adjoint divisors are constant in families.
More precisely, let $f\colon X\to T$ be a flat projective fibration
of normal varieties, where $T$ is a smooth affine curve, and suppose that
$D$ is an effective $\Q$-divisor on $X$ (whose support
does not contain any fiber of $f$) such that $K_X + D$ is $\Q$-Cartier.
Then one would like to find conditions on $f$ and $D$ which would
imply that for any ample line bundle $A$ on $X$ the nef values
$\tau _{A|_{X_t}}(K_{X_t}+D|_{X_t})$ are independent of $t$.

We know, for instance, by \cite{Wis} that if $f$ is a smooth morphism
and $D = 0$, then the nef values of the canonical class are constant in families
whenever nonnegative.
One can ask whether the same property holds more in general.
We consider the following setting

\begin{setting}\label{setting}
Let $f\colon X\to T$ be a flat projective fibration
from a normal variety $X$ to a smooth affine curve $T$.
For some $0 \in T$, assume that the fiber $X_0$ over $0$ is a normal
$\Q$-factorial variety, $N^1(X/T)\to N^1(X_0)$ is
surjective, and $D$ is an effective $\Q$-divisor on $X$ whose support
does not contain $X_0$ such that $(X_0,D|_{X_0})$
is a Kawamata log terminal pair with canonical singularities.
\end{setting}

\begin{question}\label{q:nef-value}
With the assumptions as in Setting~\ref{setting},
suppose that $K_{X_0} + D|_{X_0}$ is not nef,
and let $A$ be an ample line bundle on $X$.
Under which additional conditions can one insure that
the nef value $\tau _{A|_{X_t}}(K_{X_t}+D|_{X_t})$ is constant for
$t$ in a neighborhood of $0 \in T$?
\end{question}

As previously mentioned, a positive answer to this question
is known in the case $f$ is smooth and $D=0$ by the result of \cite{Wis}.
However, already the simple example discussed in Remark~\ref{rmk:F_2} shows that
in general one needs some additional conditions.
In this context, the following are two
natural conditions that we will consider in what follows:
\begin{enumerate}
\item
${\rm SBs}(K_X+D)$ does not contain any component of ${\rm Supp}(D|_{X_0})$, or
\item
$D|_{X_0}-aK_{X_0}$ is nef for some $a>-1$.
\end{enumerate}
We will see, in Corollary~\ref{cor:nef-value} that such conditions
are in fact sufficient to guarantee that
the nef value $\tau _{A|_{X_t}}(K_{X_t}+D|_{X_t})$ is constant
when the relative dimension of $f$ is at most $3$.

The following example shows, on the other hand, that in higher dimension
one needs further additional conditions to establish the constancy
of the nef values.

\begin{eg}\label{eg:Huybrechts}
Let $X_0$ and $X_0'$ be two
birational hyperk\"ahler manifolds of dimension $4$. By \cite{Huy03},
there exist smooth families $f\colon X\to T$ and
$f'\colon X'\to T$ that are isomorphic over $T \smallsetminus \{0\}$
and have central fibers $X_0$ and $X'_0$
respectively. Let $H'\geq 0$ be an $f'$-ample line bundle on $X'$ and $H$ its
strict transform on $X$. For any $0<\epsilon\ll 1$, the pair $(X,\epsilon H)$ is
terminal and it induces a flipping contraction $\psi \colon X\to Z$ which is an
isomorphism on the generic fiber.
\end{eg}

There are however several interesting situations where
the constancy of the nef values holds.
We will give two instances of this.
Before we can state our result, we need to introduce the right notions.
We start with the following definition
(for the definition of {\it volume} $\Vol(\x)$ of
a class $\x \in N^1(X)$ on a projective variety, we refer
to \cite[Section~2.2.C]{Laz}).

\begin{defi}\label{defi:vol-criterion-ampleness}
A projective variety $X$ is said to satisfy the
{\it volume criterion for ampleness} if the ampleness of
any given class $\x_0 \in N^1(X)$ is characterized by the existence
of an open neighborhood $U \subset N^1(X)$ of $\x_0$ (in the Euclidean topology)
such that $\Vol(\x)= \x^{\dim X}$ for every $\x \in U$.
\end{defi}

It was proven in \cite{HKP} that every projective toric variety
satisfies the volume criterion for ampleness, and
Lazarsfeld raised the question whether
the criterion holds for other classes of projective varieties, or for particular
types of divisors.
It is immediate to see that the volume criterion for ampleness
implies the criterion for ampleness via asymptotic growth of cohomologies
given in \cite{dFKL}, which is known to hold for all projective varieties.
In the case of smooth surfaces, the volume criterion
for ampleness is an elementary consequence of
the Zariski decomposition, but for which classes of higher dimensional varieties
it holds remains unknown.

The case of hyperk\"ahler manifolds gives
examples where the criterion does not hold. For instance,
if $X$ and $X'$ are two birational hyperk\"ahler manifolds of dimension
$4$, then they are isomorphic in codimension $2$ so that for every ample
class $H$ on $X$ one has $\Vol(H)=\Vol(H')=(H')^4$,
where $H'$ is the proper transform of $H$ on $X'$.
On the other hand, Section~1 and the proof of Theorem~4.6 of \cite{Huy97}
show that we have $(H')^4=H^4$, and therefore
$X$ does not satisfy the volume criterion for ampleness.
In fact, more generally, it follows by \cite[Proposition~4.12]{Bou04}
that on any hyperk\"ahler manifold $X$ the volume criterion
fails as soon as the {\it moving cone} $X$,
namely, the closure of the cone in $N^1(X)$ generated by movable divisors,
is strictly larger than the nef cone.

It is possible that, in the adjoint setting, suitable positivity assumptions
may give a better behavior. The following seems plausible.

\begin{question}
Does the volume criterion for ampleness hold for Fano varieties?
More generally, one can ask whether,
under the correct assumptions on singularities,
the criterion holds within the
the set $K_X + \Amp(X)$, where $\Amp(X)$ denotes the ample cone; or in other words,
whether it is true that a class $\x_0 \in K_X + \Amp(X)$ belongs to $\Amp(X)$ if and only
if there is an open neighborhood $U \subset N^1(X)$ of $\x_0$
such that $\Vol(\x) = \x^{\dim X}$ for every $\x \in U$.
\end{question}

The next definition will provide an alternative condition
to ensure the constancy of the nef values.

\begin{defi}\label{def:1-can}
A log pair $(X,D)$ is said to be {\it 1-canonical}
if $\ord_E(K_{X'/X} - g^*D) \ge 1$
for any log resolution $g \colon X' \to X$ of $X$
and any prime exceptional divisor $E$ on $X'$.
\end{defi}

Note that every normal variety with terminal singularities and such that
$K_X$ is Cartier is automatically 1-canonical
(that is, the pair $(X,0)$ is 1-canonical).

We now come back to the main setting of this section.
With the assumptions as in Setting~\ref{setting}, suppose that
$$
\psi \colon X\to Z
$$
is the contraction over $T$
of a $(K_X +D)$-negative extremal ray that is nontrivial
on the fiber $X_0$, and let $Z_0 = \psi(X_0)$.
By Theorem~\ref{p-1}, we know that
$$
\psi_0 := \psi|_{X_0} \colon X_0 \to Z_0
$$
is the contraction of a
$(K_{X_0} +D|_{X_0})$-negative extremal ray.

The following is the main result of this section.

\begin{thm}\label{t-nef}
With the above assumptions,
suppose that one of the following situations occurs:
\begin{enumerate}
\item
$\psi_0$ is of fiber type.
\item
$\psi_0$ is divisorial, and either
\begin{itemize}
\item
${\rm SBs}(K_X+D)$ does not contain any component of
${\rm Supp}(D|_{X_0})$, or
\item
$D|_{X_0}-aK_{X_0}$ is nef for some $a>-1$.
\end{itemize}
\item
$\psi_0$ is small, and either
\begin{enumerate}
\item[(i)]
$\psi_0$ has relative dimension $1$ (e.g., $\dim X_0 \le 3$), or
\item[(ii)]
$\dim X_0 = 4$ and $(X_0,D|_{X_0})$ is 1-canonical
(e.g., $D=0$ and $K_{X_0}$ is Cartier), or
\item[(iii)]
$X_0$ satisfies the volume criterion for ampleness (e.g., $X_0$ is toric).
\end{enumerate}
\end{enumerate}
Then the restriction of $\psi$ to
the generic fiber $X_\eta$ of $f$ is not an isomorphism.
\end{thm}

\begin{proof}
The cases in which $\psi _0$ is divisorial or of fiber type
follow from Theorem~\ref{p-1}.
We may therefore assume that $\psi _0$ is a flipping contraction.
We proceed by contradiction, and suppose that
$\psi |_{X_\eta}$ is an isomorphism.
After shrinking $T$, we may assume that $\psi$ is an isomorphism
on the complement of $X_0$.
After cutting down by the right number of
general hyperplane sections of $Z$, we can also assume
without loss of generality that $\dim \psi ({\rm Ex}(\psi ))=0$.
If $H$ is an ample divisor on $Z$,
then one sees that as $\psi$ is small,
for any $s\gg 1$, ${\rm SBs}(K_{X}+D+s\psi^*H)$
does not contain any components of ${\rm Supp}(D|_{X_0})$.

We fix an integer $s \gg 1$, and let
$$
L := K_{X}+D + s\psi^*H.
$$
After perturbing $H$, we can assume that $L$
is the multiple of a sufficiently general element of $N^1(X)_\Q$.
In the following, we let $m$ be a sufficiently divisible positive integer.
The key observation is that, by Theorem~\ref{t-extf}, the dimension
$h^0(\O _{X_t}(mL))$ is constant with respect to $t \in T$.

If $X_0$ satisfies the volume criterion for ampleness,
then we obtain immediately a contradiction. Indeed for $t \ne 0$
and some sufficiently small open neighborhood $U \subset N^1(X/T)$
of $[L]$, we have
$$
\Vol(\x|_{X_0}) = \Vol(\x|_{X_t}) = (\x|_{X_t})^{\dim X_t} = (\x|_{X_0})^{\dim X_0}
\ \fall \ \x \in U
$$
since $L|_{X_t}$ is ample, which contradicts the volume criterion on $X_0$.

Suppose now that the fibers of $\psi$ have dimension $\le 1$.
We observe that the Euler characteristic $\chi (\O _{X_t}(mL))$
is also constant with respect to $t \in T$, by flatness.
Thus, to prove the theorem in this case, it suffices to observe that
\begin{equation}\label{eq:chi+}
\sum_{i\ge 1} (-1)^i h^i(\O _{X_0}(mL)) \ne 0.
\end{equation}
This holds for the following reason. For $i > 0$ and $s \gg m$, we have
$$
h^i(X_0,\O _{X_0}(mL))
= h^0(Z_0,R^i(\psi _0) _* \O _{X_0}(m(K_X+D))\otimes \O _{Z_0}(msH))\\
$$
by the Leray spectral sequence and Serre vanishing.
In particular, if the fibers of $\psi$ have dimension at most one, then
$R^i(\psi _0) _* \O _{X_0}(m(K_X+D)) = 0$ for $i \ge 2$,
which implies that $h^i(\O _{X_0}(mL)) = 0$ for $i \ge 2$.
On the other hand, as $K_X+D$ is relatively anti-ample,
we have $R^1(\psi _0) _* \O _{X_0}(m(K_X+D)) \ne 0$,
and since this sheaf is supported on a zero dimensional set,
it follows that $h^1(\O _{X_0}(mL)) \ne 0$.
Equivalently, one can use the main result of \cite{dFKL} to deduce
that, for $m$ sufficiently divisible, the space $H^1(\O _{X_0}(mL))$ is nontrivial.
Therefore \eqref{eq:chi+} holds in this case.

It remains to consider the case when $X_0$ is a 4-dimensional variety
with 1-canonical singularities.
Since we have already excluded the case when $\psi$ has fibers of dimension $\le 1$,
the only possibility is that
$\phi$ is a $(K_X + D)$-flipping contraction with exceptional locus $\Ex(\phi)$
of dimension $2$.
Let
$$
\xymatrix@C=15pt@R=15pt{
X \ar[dr]_\phi \ar@{-->}[rr]^\psi && X^+ \ar[dl]^{\phi^+} \\
& Z
}
$$
be the flip over $T$, which exists by \cite{BCHM}, and let
$D^+$ be the proper transform of $D$ on $X^+$.
If $\Ex(\phi^+)$ is the exceptional locus of $\phi^+$, then we have that
$$
\dim \Ex(\phi) + \dim \Ex(\phi^+) \geq \dim X -1 =4
$$
by \cite[Lemma 5.1.17]{KMM87}, and therefore $\dim \Ex(\phi^+) \in \{2, 3\}$.

By Theorem~\ref{p-1}, $\psi$ restricts to a $(K_{X_0} + D|_{X_0})$-flip
$\psi_0 \colon X_0 \rat X_0^+$, where $X_0^+ \subset X^+$ is the fiber
of $\phi^+$ over $0$. Note that $\Ex(\phi_0) = \Ex(\phi)$, and similarly
$\Ex(\phi_0^+) = \Ex(\phi^+)$. In particular, $\dim \Ex(\phi^+) = 2$.
We fix a common resolution $Y$ of $\psi_0$,
with maps $p \colon Y \to X_0$ and $q\colon Y \to X^+_0$.

Let $W \subset X^+$ be a maximal dimensional irreducible component
of $\Ex(\phi^+)$. Since $\codim(W,X_0^+) = 2$, the minimal log discrepancy
of $X_0^+$ at the generic point of $W$ is at most $2$
(cf. \cite[Main Theorem~1]{Ambro}) and so a fortiori is the minimal log discrepancy
of $(X_0^+,D^+|_{X_0^+})$.
Recalling that minimal log discrepancies on log terminal
varieties can be computed from log resolutions,
this implies that there is a
prime exceptional divisor $F$ over $X^+_0$ (that we may
assume lying on $Y$) such that
$$
\ord_F(K_{Y/X^+_0} - q^*D^+|_{X_0^+}) \le 1.
$$
Since $\psi$ is a $(K_{X_0} + D|_{X_0})$-flip, this implies that
$$
\ord_F(K_{Y/X_0} - p^*D|_{X_0}) < 1.
$$
This contradicts the hypothesis that $(X_0,D|_{X_0})$ is 1-canonical.
The proof of the theorem is now complete.
\end{proof}

We have the following immediate consequence.

\begin{cor}\label{cor:nef-value}
With the same assumptions as in Setting~\ref{setting},
suppose that $K_{X_0} + D|_{X_0}$ is not nef,
and let $A$ be an ample line bundle on $X$.
Assume that we are in one of the following cases:
\begin{enumerate}
\item[(i)]
$\dim X_0 \le 3$.
\item[(ii)]
$\dim X_0 = 4$ and $(X_0,D|_{X_0})$ is 1-canonical
(e.g., $D=0$ and $K_{X_0}$ is Cartier).
\item[(iii)]
$X_0$ satisfies the volume criterion for ampleness (e.g., $X_0$ is toric).
\end{enumerate}
Then the nef value $\tau _{A|_{X_t}}(K_{X_t}+D|_{X_t})$ is constant for
$t$ in a neighborhood of $0 \in T$.
\end{cor}

\begin{rmk}\label{rmk:hyperkahler}
Example~\ref{eg:Huybrechts} shows how
both Theorem~\ref{t-nef} and Corollary~\ref{cor:nef-value}
fail in general if one does not impose any additional hypothesis
to Setting~\ref{setting}, and also that
case~(ii) of both statements is sharp.
\end{rmk}

\section{Deformations of Mori chamber decompositions of Fano varieties}

It follows by results in \cite{BCHM} that any Fano variety $X$
with $\Q$-factorial log terminal singularities
is a {\it Mori dream space} in the sense of \cite{HK}.
In particular, the {\it moving cone} $\Mov^1(X) \subset N^1(X)$,
namely, the closure of the cone generated by movable divisors,
admits a finite decomposition into polyhedral cones,
called {\it Mori chamber decomposition}.
One of these chambers is the nef cone $\Nef(X)$ of $X$.
In general, a {\it Mori chamber} of $\Mov^1(X)$
is the closure of a Mori equivalence class
whose interior is open in $N^1(X)_\Q$, where we
declare that two elements $L_1$ and $L_2$ of
$N^1(X)_{\Q}$ are {\it Mori equivalent}
if $\Proj R(L_1) \cong \Proj R(L_2)$.

The Mori chamber decomposition of a Fano variety
retains information on the biregular and
birational geometry of $X$ both in terms of the log minimal models
obtainable from $X$ via suitable log minimal model programs.
Wall-crossing between two adjacent chambers of maximal dimension
corresponds to small modifications between corresponding
log minimal models.

Throughout this section, we consider a flat projective morphism
$$
f \colon X \to T
$$
onto a smooth curve $T$, whose fibers $X_t$ are Fano varieties
with $\Q$-factorial terminal singularities.

When $f$ is a smooth family of Fano manifolds,
then it follows by the result of Wi\'sniewski on nef values
that the nef cone and the Mori cone are locally constant
under small deformations of Fano manifolds \cite{Wis,Wis2}.
Here we are interested in the behavior of Mori chamber decompositions
of the fibers. We consider the following question.

\begin{question}\label{q:mori-chamber-dec}
With the above notation, under which conditions is the Mori chamber decomposition
of the fibers of $f$
locally constant over $T$ (in the analytic or \'etale topology)?
\end{question}

\begin{rmk}
In an earlier version of this paper, we had conjectured that the Mori chamber decomposition
of the fibers of $f$ is always locally constant without any additional condition.
However, this was disproved by Totaro (cf. \cite{Tot09}).
\end{rmk}

If we want study the behavior under deformation of the whole Mori chamber
decomposition, then it seems that allowing singularities
is a necessary generalization,
as we will need to apply steps of the minimal model program.

\begin{rmk}\label{eg-controexamples}
The example of a family of quadrics $\P^1 \times \P^1$ degenerating
to an irreducible quadric cone shows that the nef cone may jump
if one relaxes the assumptions on singularities from terminal to canonical.
The example of the
family of quadrics $\P^1 \times \P^1$ degenerating to an $\F_2$
discussed in Remark~\ref{rmk:F_2} shows
that one cannot hope for a positive answer to the question for families of
weak Fano varieties (even in the smooth case) or families of
log Fano varieties (even in the log-smooth case).
\end{rmk}

\begin{rmk}
In the Zariski topology, in general there
are no natural isomorphisms $N^1(X_t) \to N^1(X_u)$
and $N_1(X_t) \to N_1(X_u)$ for $t \ne u$ in $T$
unless one first fixes a path joining $t$ to $u$.
This is the case, for instance, for a quadric fibration $f \colon X \to T$
with all fibers isomorphic to a smooth quadric $\P^1\times\P^1$
and relative Picard number $\rho(X/T) = 1$
(an explicit example is given by the family of quadrics of equation
$\{xy + z^2 + tu^2 = 0\} \subset \P^3 \times \C^*$, where $(x,y,z,u)$
are the homogeneous coordinates on $\P^3$ and $t \in \C^*$).
\end{rmk}

One can restate Question~\ref{q:mori-chamber-dec} more precisely,
as we will see below.
We know by Corollary~\ref{cor:factoriality}
that a flat projective family over a smooth curve $f \colon X \to T$
of Fano varieties with terminal $\Q$-factorial singularities
satisfies the conditions in \cite[(12.2.1)]{KM92}.
We can therefore consider the local systems $\GN^1(X/T)$ and $\GN_1(X/T)$
introduced in \cite[Section~12]{KM92} (these local systems are defined
\cite{KM92} using rational coefficients; they will be considered here
with real coefficients).
These are sheaves on $T$ in the analytic topology.
For any analytic open set $U \subseteq T$, these are given by:
$$
\GN^1(X/T)(U) =
\left\{\text{sections of ${\mathcal N}^1(X/T)$ over $U$ with open support}
\right\},
$$
where, in our situation, ${\mathcal N}^1(X/T)$ is the
functor given by $N^1(X\times_TT'/T')$
for any $T' \to T$, and
$$
\GN_1(X/T)(U) =
\left\{
\substack{\displaystyle\text{flat families of 1-cycles
$C/U \subseteq X\times_TU$ with real}\\
\displaystyle\text{coefficients, modulo fiberwise numerical equivalence}}
\right\}.
$$
Note that, in our setting, every nonzero section of ${\mathcal N}^1(X/T)$
on an open set $U \subseteq T$ has open support.
It is shown in \cite[(12.2)]{KM92} that
$\GN^1(X/T)$ and $\GN_1(X/T)$ are dual local
systems with finite monodromy and, moreover, that
$\GN^1(X/T)|_t = N^1(X_t)$ and $\GN_1(X/T)|_t = N_1(X_t)$
for very general $t \in T$.
Applying~\cite[(12.1.1)]{KM92}, we obtain the following property.

\begin{prop}\label{p-N_1}
With the above assumptions, we have
$$
\GN^1(X/T)|_t = N^1(X_t) \ \and \ \GN_1(X/T)|_t = N_1(X_t)
$$
for all $t \in T$.
In particular, the Picard number $\rho(X_t)$ is independent of $t \in T$.
\end{prop}

\begin{proof}
After taking a finite base change, we can assume that
$\GN^1(X/T)$ and $\GN_1(X/T)$ have trivial monodromy, so that,
in particular, there are natural identifications
between the fibers of $\GN^1(X/T)$ (resp. $\GN_1(X/T)$)
and $N^1(X/T)$ (resp. $N_1(X/T)$).

If $t \in T$ is very general, then the natural maps
$N^1(X/T) \to N^1(X_t)$ and $N_1(X_t) \to N_1(X/T)$
are isomorphisms, since, as we have already mentioned,
we have $\GN^1(X/T)|_t = N^1(X_t)$ and $\GN_1(X/T)|_t = N_1(X_t)$.
The following lemma implies that in fact this holds for a general $t \in T$.

\begin{lem}\label{lem:T^o}
There exists a nonempty open subset
$T^\o \subseteq T$ such that the natural maps
$N^1(X/T) \to N^1(X_t)$ and $N_1(X_t) \to N_1(X/T)$
are isomorphisms for every $t \in T^\o$.
\end{lem}

\begin{proof}
By Verdier's generalization
of Ehresmann's theorem \cite[Corollaire~(5.1)]{Ver}, there is a
nonempty open set $T^\o \subseteq T$ such that the restriction
$f^\o \colon X^\o \to T^\o$ of $f$ to $X^\o := f^{-1}(T^\o)$
is a topologically locally trivial fibration.

If $t \in T^\o$ is very general, then we know
that $N^1(X/T)_\Q \to N^1(X_t)_\Q$ is an isomorphism.
On the other hand, $\r(X_t)$ is constant for $t \in T^\o$,
since the fibers are Fano varieties and the fibration is
topologically locally trivial.
So to conclude, it suffices to show that $N^1(X/T)_\Q \to N^1(X_t)_\Q$
is surjective for every $t \in T^\o$.

Fix an arbitrary $t \in T^\o$, and let $\D \subseteq T^\o$ be
a contractible analytic neighborhood of $t$.
Let $f_\D \colon X_\D := f^{-1}(\D) \to \D$ be the restriction of $f$.
Since $f_\D$ is topologically locally trivial and $\D$ is contractible,
the restriction map $H^2(X_\D, \Z) \to H^2(X_t,\Z)$ is an isomorphism.
In particular, if $L_t$ is a line bundle on $X_t$, then $c_1(L_t)$
extends to a cycle $\g \in H^2(X_\D, \Z)$. The restriction $\g|_{X_u}$ of $\g$
to any other fiber $X_u$ of $f_\D$ is equal to the first Chern class of
some line bundle $L_u$, since $\Pic(X_u) \cong H^2(X_u,\Z)$.
As we can take $u \in \D$ to be a very general point of $T^\o$,
we can find a class $\x \in N^1(X/T)_\Q$ restricting to $[L_u]$.
After re-scaling, we can assume that $\x = [L]$ for some line bundle on $X$.
Using a topological trivialization $X_\D \approx X_u \times \D$
that induces an isomorphism $H^2(X_\D,\Z) \cong H^2(X_u,\Z)$
sending $\g$ to $\g|_{X_u}$, we see that $c_1(L)|_{X_\D} = \g$.
This implies that $c_1(L|_{X_t}) = c_1(L_t)$, and
hence that $L|_{X_t} = L_t$.

This proves that $N^1(X/T) \to N^1(X_t)$ is an isomorphism for every $t \in T^\o$.
The statement on $N_1(X_t) \to N_1(X/T)$ follows by duality.
\end{proof}

Back to the proof of the proposition,
we fix now an arbitrary point $0 \in T$.
If $0 \in T^\o$, then there is nothing to prove.
Suppose otherwise that $0 \not\in T^\o$.
After shrinking $T$, we may assume that $T^\o = T \setminus \{0\}$.
We also assume that $T$ is affine.
Shrinking further $T$ around $0$ if necessary, we can find a log resolution
$$
g \colon Y \to X
$$
of $(X,X_0)$, with the property that $N^1(Y^\o/T^\o) \to N^1(Y_t)$
is an isomorphism for all $t \in T^\o$, where $Y^\o = g^{-1}(X^\o)$.
By taking general complete intersections of very ample divisors forming
a basis of $N^1(Y^\o/T^\o)$, we can find families
of curves $C_{1,t}, \dots, C_{r,t} \subseteq Y_t := g^{-1}(X_t)$,
dominating $T^\o$, whose classes generate $N_1(Y_t)$
for every $t \in T^\o$. We deduce that any curve in a fiber of $Y \to T$
is numerically equivalent to a 1-cycle supported inside $Y_0$.

By construction, $Y_0$ is a divisor with simple normal crossing support.
After taking a base change, we may also assume that $Y_0$ is reduced.
Denote by $g_0 \colon Y_0 \to X_0$ the restriction of $g$, and
consider the commutative diagram
$$
\xymatrix{
N^1(Y/T) \ar[rr]^{r'} && H^2(Y_0,\R) \\
N^1(X/T) \ar@{^{(}->}[u]^{g^*} \ar[r]^r & N^1(X_0) \ar@{=}[r]
& H^2(X_0,\R) \ar@{^{(}->}[u]^{g_0^*}.
}
$$
The sheaf $\GN^1(Y/T)$ is a local system with trivial monodromy,
and thus is a locally constant sheaf. It follows that if $\D \subset T$ is
a contractible analytic neighborhood of $0$ and $Y_\D \subset Y$ is its inverse image, then
there are natural identifications (the second one induced by restriction)
$$
N^1(Y_{\D}/\D) = \GN^1(Y/T)(\D) = \GN^1(Y/T)(T) = N^1(Y/T).
$$
Therefore we can apply~\cite[(12.1.1)]{KM92},
which says that the restriction map $r'$ is an isomorphism.
We need to show that $r$ is an isomorphism as well.

If $\a \in N^1(X_0) = H^2(X_0,\R)$ is an integral point,
then we have $g_0^*\a = r'([L])$ for some line bundle $L$ on $Y$.
Observe that $L|_{Y_0}$ is numerically trivial over $X_0$.
Since any curve in a fiber of $g$
is numerically equivalent to a 1-cycle supported on a fiber of $g_0$, it
follows that $L$ is numerically trivial over $X$.
Considering $L$ as a divisor, we take the push-forward
$g_*L$, which is $\Q$-Cartier since $X$ is $\Q$-factorial.
Applying the negativity lemma to both $L - g^*g_*L$
and its opposite, we conclude that $L = g^*g_*L$.
It follows then by the injectivity of $g_0^*$ and the commutativity
of the diagram that $r([g_*L])= \a$. This proves that $r$ is surjective.

Let $\x \in N^1(X/T)$ be any nonzero element. Since $\x \neq 0$, there is a
curve $C$ in a fiber of $f$ such that $\x\.C \ne 0$.
Since any curve $C'$ on $Y$ mapping to $C$
is numerically equivalent to a 1-cycle supported inside $Y_0$,
it follows that $C$ is numerically equivalent to a 1-cycle (with rational
coefficients) $\g$ supported inside $X_0$. Since
$r(\x)\.\g = \x \. C \ne 0$, we conclude that $r(\x) \ne 0$.
This shows that $r$ is injective.

Therefore $N^1(X/T) \to N^1(X_0)$ is an isomorphism, and hence,
by duality, $N_1(X_0) \to N_1(X/T)$ is an isomorphism as well.
This proves the proposition.
\end{proof}

Using these local systems, the property sought
in Question~\ref{q:mori-chamber-dec} (including its
consequences on the behavior under deformations of
nef cones and Mori cones) can be restated as follows.
Suppose that $f \colon X \to T$ is a family satisfying the conclusions
of Question~\ref{q:mori-chamber-dec}.
Let $\r$ be the Picard number of a (equivalently, any) fiber of $f$.
Then there is a local system $\GG\S$ on $T$, with fibers equal to
a finite polyhedral decomposition $\S$ of a cone in $\R^\r$
(with a forgetful morphism $\S \to \R^\r$), and a map of local systems
$\GG\S \to \GN^1(X)$, such that the induced maps of fibers
$$
\S = \GG\S|_t \to \GN^1(X/T)|_t = N^1(X_t)
$$
gives the Mori chamber decomposition of $\Mov^1(X_t)$ for every $t \in T$.
In particular, there are local subsystems of cones
$\GNef(X/T) \subset \GN^1(X/T)$ and $\GNE(X/T) \subset \GN_1(X/T)$ with fibers
$$
\GNef(X/T)|_t = \Nef(X_t)
\ \and \
\GNE(X/T)|_t = \NE(X_t)
$$
for every $t \in T$.

\begin{rmk}
A positive answer to Question~\ref{q:mori-chamber-dec} would imply that,
for every $t,u \in T$ and every path $\gamma$ from $t$ to $u$,
there are natural isomorphisms $N^1(X_t) \to N^1(X_u)$ and $N_1(X_t) \to N_1(X_u)$
(depending on $\gamma$) compatible with the
above chamber decompositions and cones.
\end{rmk}

It was observed by Musta\c t\u a and Lazarsfeld that, as
a direct application of extension theorems,
in the hypotheses of the conjecture,
the pseudo-effective cones of the fibers of $f$ are locally constant in the family.
In other words, using the formalism introduced above, there is a
local subsystem of cones $\GPEff(X/T) \subset \GN^1(X/T)$ with fiber
$\GPEff(X/T)|_t = \PEff(X_t)$ for every $t \in T$.
For this result, one only needs a small generalization of Siu's invariance of
plurigenera, which is well-known (equivalently, one can apply Theorem~\ref{t-extf}).

In a similar vein, we have the following result for the moving cone.
Note that the Mori chamber decompositions are supported on the moving cones.

\begin{thm}\label{prop:movable-cones}
Let $f \colon X\to T$ be a flat projective family over a smooth curve
of Fano varieties with $\Q$-factorial terminal singularities.

Then there is a local subsystem of cones $\GMov^1(X/T) \subset \GN^1(X/T)$ with fiber
$$
\GMov^1(X/T)|_t = \Mov^1(X_t),
$$
the moving cone of $X_t$, for every $t \in T$.
\end{thm}

\begin{proof}
After base change, we can assume without loss of generality that
$\GN^1(X/T)$ has trivial monodromy, and thus the natural homomorphism
$N^1(X/T) \to N^1(X_t)$ is an isomorphism for every $t \in T$
by Proposition~\ref{p-N_1}.

Since by Theorem~\ref{t-extf} all the sections of the restriction to $X_0$
of any line bundle $L$ extend to $X$, it follows that if $|L|_{X_0}|$
is a movable linear system, then so is $|L|_{X_t}|$ for every $t$ near $0$.
Thus, to prove the proposition, we need to show that if $L$ is a relatively big line
bundle whose restriction $L|_{X_t}$ is in the interior of $\Mov^1(X_t)$
for every $t \ne 0$, then $L|_{X_0}$ is movable as well.

Suppose otherwise that $L|_{X_0}$ is not movable.
After perturbing $L$ and re-scaling, we may assume that $L|_{X_0} \not\in \Mov^1(X_0)$.
We can find an effective $\Q$-divisor $D$ on $X$ such that
$K_X + D \sim_\Q \lambda L$ for some $\lambda > 0$ and
$(X,D)$ is a Kawamata log terminal pair with canonical singularities.
We fix $a \gg 0$ and run a minimal model program for $(X,D)$
directed by $D - aK_X$. On a general fiber $X_t$ ($t \ne 0$)
this minimal model program is a composition of flips, as the
stable base locus of $L|_{X_t}$ does not contain any divisor.
On the other hand, on the central fiber $X_0$ the induced
minimal model program (cf. Theorem~\ref{p-1})
must contract, at some point, the divisorial components of the stable base
locus of $L|_{X_0}$. Since $-K_{X_0}$ is relatively ample with respect to
any extremal contraction occurring in such minimal model program,
at each stage the central fiber (i.e., the proper transform of $X_0$)
has terminal singularities.
Therefore, once we reach the step where a divisor on the central fiber
is being contracted, we obtain a contradiction with part~(b) of Theorem~\ref{p-1}.
\end{proof}

A partial answer to Question~\ref{q:mori-chamber-dec} comes from the following result.
In view of Totaro's examples (cf. \cite{Tot09}), the conditions imposed
in cases~(a) or~(b) would seem to be optimal.

\begin{thm}\label{thm:mori-chamber-dec}
Question~\ref{q:mori-chamber-dec} has a positive answer in the following cases:
\begin{enumerate}
\item[(i)]
$\dim X_0 \le 3$.
\item[(ii)]
$\dim X_0 = 4$ and $X_0$ is 1-canonical (e.g., $K_{X_0}$ is Cartier).
\item[(iii)]
$X_0$ is toric.
\end{enumerate}
\end{thm}

\begin{proof}
We first observe that each property assumed in the above cases
(i)--(iii)
is preserved throughout the steps of a minimal model program of $X$ over $T$.
After base change, we can assume without loss of generality that
$\GN^1(X/T)$ has trivial monodromy and the natural homomorphism
$i_t^* \colon N^1(X/T) \to N^1(X_t)$ is an isomorphism for every $t \in T$.

After shrinking $T$ around $0$, we can assume that
the Mori chamber decomposition of the fibers of $f$ is constant
away from the central fiber,
so that there is a finite polyhedral decomposition $\S$ of a cone in
$N^1(X/T)$ which induces, for $t \ne 0$, the Mori chamber decomposition
$\S_t$ of $N^1(X_t)$.

Indeed, by running log minimal model programs for
Kawamata log terminal pairs $(X,D)$
such that $K_X + D$ is in the interior of $\Mov^1(X/T)$,
and applying Theorem~\ref{p-1},
we see that the Mori chamber decomposition of $\Mov^1(X/T)$ is a refinement of
the Mori chamber decomposition of $\Mov^1(X_t)$ for any fiber $X_t$.
On the other hand, suppose that $K_X + D$ lies on a wall between two
Mori chambers of $\Mov^1(X/T)$ while, for some $t \in T$, the restriction
$K_{X_t} + D|_{X_t}$ lies in the interior of a Mori chamber of $\Mov^1(X_t)$.
Fix a general relatively ample effective divisor $A$, so that
$K_X + D + \ep A$ is in the interior of a Mori chamber of $\Mov^1(X/T)$ for
all $0 < \ep \ll 1$, and fix such an $\ep$. Then, working over $T$,
the contraction $Z \to W$ from the minimal model $Z$ of $(X,D+\ep A)$
to the canonical model $W$ of $(X,D)$
is as isomorphism on $X_t$ for $t\in T$ chosen as above,
and hence it is an isomorphism
over a dense open subset of $T$. Using the fact
that the decomposition of $\Mov^1(X/T)$ is finite,
we can eliminate all the fibers, other than $X_0$, on which these
type of contractions are nontrivial, to reduce to the situation where the
Mori chamber decomposition of the fibers of $f$ is constant
away from the central fiber.

By Theorem~\ref{prop:movable-cones}, the cone of
movable divisors of $X_t$ is locally constant, and thus
the Mori chamber decomposition $\S_0$
of $N^1(X_0)$ is supported on the same cone
which supports the decomposition $i_0^*(\S)$ induced by $\S$ via the
isomorphism $i_0^* \colon N^1(X/T) \to N^1(X_0)$.
A priori, the Mori chamber decomposition $\S_0$
is a refinement of the decomposition $i_0^*(\S)$,
and the statement of the theorem is that the two decompositions agree.

Suppose by contradiction that $\S_0$ is finer than $i_0^*(\S)$.
Let $\D$ be an effective big $\Q$-divisor on $X$ not containing any
of the fibers of $f$ and such that $\D|_{X_t}$ is in the
interior of a Mori chamber $\M_t$ of $X_t$ if $t \ne 0$, whereas $\D|_{X_0}$
lies on a wall separating two contiguous Mori chambers
$\M_0$ and $\M_0'$ of $X_0$.
We can assume that $\D = A + B$, where $A$ is an effective ample $\Q$-divisor
and $B$ is an effective $\Q$-divisor. If $A$ is chosen generally,
then we can furthermore assume that every small perturbation $(\D + r A)|_{X_0}$,
for $r \ne 0$,
of $\D|_{X_0}$ does not lie on the wall separating $\M_0$ and $\M_0'$.

For $m$ sufficiently divisible, we fix a general $H \in |-K_X|$,
and consider the divisor $D := \frac 1m H + \D$.
After re-scaling $\D$, we can assume that $(X,D)$ is a Kawamata log terminal variety
with terminal singularities.
Note that, if $|r| \ll 1$, then $D + r A$ is effective and
$(X,D + rA)$ is a Kawamata log terminal variety with terminal singularities.

We consider a small perturbation $\D + \ep A$, where $0 < \ep \ll 1$
is a rational number.
Suppose that, for $t=0$, the Mori chamber containing $(\D + \ep A)|_{X_0}$ is $\M_0$.
Note that $(\D - \ep A)|_{X_0}$ is in the interior of $\M_0'$,
if $\ep$ is sufficiently small.

We run a minimal model program for $(X,D + \ep A)$ over $T$.
This gives a sequence of flips $\phi \colon X \rat Z$, ending with a log minimal model
$Z$ over $T$. By Theorem~\ref{p-1}, $\phi$ restricts to a
sequence of flips $\phi_t \colon X_t \rat Z_t$ on every fiber $X_t$ of $t$.
These maps induce isomorphisms $N^1(X_t) \to N^1(Z_t)$,
and the Mori chambers $\M_t$ are mapped, via such isomorphisms, to the
nef cones $\Nef(Z_t)$.

On $Z$ we consider the divisor $\G := \phi_*(D - \ep A)$.
If $\ep$ is sufficiently small,
then $(Z,\G)$ is a Kawamata log terminal pair
with canonical (in fact, terminal) singularities, and
the restriction of the stable base locus of $\G$ over $T$
to any fiber $Z_t$ does not contain any divisorial component.
Note also that the maps $N^1(Z/T) \to N^1(Z_t)$ are isomorphisms for every $t \in T$.
It follows by Corollary~\ref{cor:nef-value}
that the nef value of the restriction of this pair to any fiber
against any ample divisor on $Z$ is constant, if anywhere positive.
This however is not the case. Indeed, the divisor $K_{Z_t} + \G|_{Z_t}$
is ample for all $t \ne 0$, and is not nef if $t = 0$.
This gives a contradiction, and hence completes the proof of the theorem.
\end{proof}

As we shall see in the next section, the invariance of the Mori
chamber decomposition when $X_0$ is a toric variety
is just a hint of a much stronger rigidity property.

\section{Rigidity properties of toric Fano varieties}

This last section is devoted to the proof of the following result.
Throughout the proof, all divisors will be chosen in such a way that
the restrictions considered throughout are well defined.

\begin{thm}\label{thm:rigidity-of-toric}
Simplicial toric Fano varieties with at most terminal singularities
are rigid under small projective flat deformations.
\end{thm}

\begin{proof}
Let $X_0$ be a simplicial (and hence $\Q$-factorial)
toric Fano variety with at most terminal singularities,
and suppose that $f \colon X \to T$ is a projective flat deformation
of $X_0$ over a smooth pointed affine curve $T \ni 0$.
After shrinking $T$ near $0$,
we can assume that $X$ has terminal $\Q$-factorial singularities,
and that all fibers $X_t$ are terminal $\Q$-factorial Fano varieties.
By also taking a base change if necessary, we can furthermore assume that
$N^1(X/T) \to N^1(X_t)$ is an isomorphism for all $t \in T$
(cf. Proposition~\ref{p-N_1}).

\begin{lem}\label{lem:Cl}
After a suitable base change of $f$,
the restriction map $\Cl(X/T) \to \Cl(X_0)$ is an isomorphism.
\end{lem}

\begin{proof}
The morphism $f \colon X \to T$ can be extended to a morphism
$\ov f \colon \ov X \to \ov T$,
where $\ov X$ and $\ov T$ are completions of $X$ and $T$ into projective varieties.
We assume that $\ov X$ is normal.
Let $\ov S \subseteq \ov X$ be the intersection of $\dim X - 3$ general hyperplane
sections, and denote $S = \ov S \cap X$ and $S_0 = S \cap X_0$.
Since $X$ is terminal, the singular locus of $X$ has codimension $\ge 3$.
Thus we can assume that $S$ is smooth. In fact, by shrinking $T$ near $0$, we
can also assume that the restricted morphism $S \to T$ is a smooth family of surfaces.

By \cite[Proposition~12.2.5]{KM92}, the local system $\GN^1(S/T)$
has finite monodromy. After a suitable base change of $\ov f$, we may assume that
the monodromy is trivial, and hence that for every $t \in T$
the restriction map $N^1(S/T) \to N^1(S_t)$ is an isomorphism.
Observe that $H^1(\O_{S_t}) = 0$ for every $t$,
since $S_t$ is a complete intersection of hyperplane sections of $X_t$
and $H^1(\O_{X_t}) = 0$. We deduce that there is an injection
$\Pic(S_t) \inj H^2(S_t,\Z)$
whose cokernel is contained in $H^2(\O_{S_t})$, and hence is torsion free.

Since $X_0$ is toric, the class group $\Cl(X_0)$ is finitely generated;
we fix a finite set of generators.
Let $D$ be any of the selected generators, and
consider its class $\d = c_1(\O_{S_0}(D|_{S_0})) \in H^2(S_0,\Z)$.
Note that, for some integer $m \ge 1$,
there is a line bundle $\LL$ on $X$ such that $c_1(\LL|_{S_0}) = m\d$.
Since $S \to T$ is smooth, and thus topologically
locally trivial, for every $t \in T$
there is an isomorphism $h_\g \colon H^2(S_0,\Z) \to H^2(S_t,\Z)$,
possibly depending on a path $\g$ joining $0$ to $t$.
By construction, we have $c_1(\LL|_{S_t}) = h_\g(m\d)$,
and thus $h_\g(m\d)$ is in the image of $\Pic(S_t)$.
Since the cokernel of $\Pic(S_t) \inj H^2(S_t,\Z)$ is torsion free, this implies that
$h_\g(\d)$ is in the image of $\Pic(S_t)$.
We fix a divisor $B_t$ on $S_t$ such that $c_1(\O_{S_t}(B_t)) = h_\g(\d)$.
Using an Hilbert space argument, we conclude
that, if $t$ is very general, then the divisor $B_t$ belongs to a (one
dimensional) algebraic family dominating $T$. After taking a base change,
we can assume that $B_t$ moves in a family parametrized by $T$.
We obtain in this way a divisor $B$ on $S$ such that
$c_1(\O_{S_0}(B|_{S_0})) = c_1(\O_{S_0}(D|_{S_0}))$, and thus
$B|_{S_0} \sim D|_{S_0}$.

Since in this process we are only considering finitely many divisors
(namely, the selected generators of $\Cl(X_0)$),
we conclude that, after a suitable finite base change,
the image of the restriction map $\Cl(S/T) \to \Cl(S_0)$ contains the image of
$\Cl(X_0) \to \Cl(S_0)$. Note that both restriction maps are injective:
the injectivity of $\Cl(S/T) \to \Cl(S_0)$ follows by the fact that $S \to T$ is
a smooth morphism to an affine curve,
and the injectivity of $\Cl(X_0) \to \Cl(S_0)$ by the main theorem of \cite{RS06}.
Taking closure in $\ov S$ of divisors on $S$ gives a splitting of the surjection
$\Cl(\ov S/\ov T) \to \Cl(S/T)$.
Moreover, applying again the main theorem of \cite{RS06}, we see that the restriction map
$\Cl(\ov X) \to \Cl(\ov S)$ is an isomorphism.
Altogether, we have a commutative diagram
$$
\xymatrix{
\Cl(\ov X/\ov T) \ar[d]_\cong \ar[r] &\Cl(X/T) \ar[d]\ar[r] &\Cl(X_0) \ar@{^{(}->}[d] \\
\Cl(\ov S/\ov T) \ar@{->>}[r] &\Cl(S/T) \ar[r]^(.32){\cong} \ar@/_15pt/[l]
&\im\big(\Cl(S/T) \to \Cl(S_0)\big),
}
$$
which shows that the map $\Cl(X/T) \to \Cl(S/T)$ is surjective.
One observes that this map is also injective (and hence an isomorphism),
since, again by \cite{RS06}, it induces an injection
$\Cl(X_t) \to \Cl(S_t)$ for every $t \in T$.
We conclude by the diagram that $\Cl(X/T) \to \Cl(X_0)$ is an isomorphism.
\end{proof}

We consider the total coordinate ring $R_0$ of $X_0$ (cf. \cite{Cox}).
This ring, which is defined in terms of the combinatorial data
attached to the fan defining the toric variety $X_0$, can
equivalently be described as
$$
R_0 = \bigoplus_{[D] \in \Cl(X_0)} H^0(\O_{X_0}(D)).
$$
This is a polynomial ring, with product compatible with the
multiplication maps
$$
H^0(\O_{X_0}(D)) \otimes H^0(\O_{X_0}(D')) \to H^0(\O_{X_0}(D+D')).
$$
If $\S$ is the fan attached to the toric variety, then
$$
R_0 = \C[x_{0,1},\dots,x_{0,r}],
$$
where each variable $x_{0,i}$ corresponds to a ray of $\S$
and is identified with the primitive generator of the ray, which
defines a toric invariant divisor on $X_0$. We will denote
such divisor by $\Div(x_{0,i})$.

By the lemma, after suitable base change, the restriction map
$\Cl(X/T) \to \Cl(X_0)$ is an isomorphism. We consider the ring
$$
R_t = \bigoplus_{[A] \in \Cl(X/T)} H^0(\O_{X_t}(A)).
$$
Note that for $t=0$ this gives the ring $R_0$ previously defined.

There is a natural $\Cl(X_0)$-grading on the ring $R_0$ (cf. \cite{Cox}), or equivalently,
a $\Cl(X/T)$-grading. We fix an ample divisor $H$ on $X$. By taking intersections
with $H^{\dim X_0 -1}$, we obtain a $\Z$-grading on $\Cl(X_t)$ for each $t$, and hence
on $\Cl(X/T)$. For every divisor $A$ on $X$ and every $t$, we denote
$$
\deg(A) = \deg (A|_{X_t}) := A|_{X_t} \.H^{\dim X_t-1}.
$$
This gives a $\Z$-grading on each $R_t$.
According to this grading, $R_0$ is a (positively) weighted polynomial ring. We set
$$
m = \min_{1\le i \le r} \deg(\Div(x_{0,i})),\quad
M = \max_{1\le i \le r} \deg(\Div(x_{0,i})).
$$

Since $-K_X$ is relatively ample over $T$,
it follows by Theorem~\ref{t-extf} that, for every $[A] \in \Cl(X/T)$
and every $t \in T$, the restriction map $H^0(\O_X(A)) \to H^0(\O_{X_t}(A))$
is surjective. For every integer $d \ge 0$, we consider the locally free sheaf
$$
\E^{\le d} := \bigoplus_{\substack{[A]\in \Cl(X/T) \\ \deg (A) \le d}} f_*\O_X(A).
$$
Let $E^{\le d}$ be the associated vector bundle. The fiber of
$E^{\le d}$ over $t \in T$ is given by
$$
E^{\le d}|_t := \bigoplus_{\substack{[A]\in \Cl(X/T) \\ \deg (A) \le d}}H^0(\O_{X_t}(A)),
$$
which is a direct summand of $R_t$.
Note that $E^{\le d}$ is a trivial vector bundle on $T$.

The elements $x_{0,i} \in R_0$, each thought as a section of the appropriate
sheaf, deform away from the central fiber, to elements $x_{t,i} \in R_t$.
For every $t \in T$, we consider the homomorphism of $\C$-algebras
$$
\Phi_t \colon \C[x_1,\dots,x_r] \to R_t, \quad x_i \mapsto x_{t,i}.
$$
We claim that this map is an isomorphism of $\C$-algebras for infinitely many $t \in T$.
Note that, to this end, it suffices to check that $\Phi_t$ is bijective.

For every $e,d \ge 0$, the homomorphism $\Phi_t$ induces a linear map of vector spaces
$$
\Phi_t^{e,d} \colon \C[x_1,\dots,x_r]^{\le e}
\to \bigoplus_{\substack{[A]\in \Cl(X/T) \\ \deg (A) \le d}}H^0(\O_{X_t}(A)),
$$
where $\C[x_1,\dots,x_r]^{\le e}$ is the subspace of $\C[x_1,\dots,x_r]$ generated
by the monomials of degree $\le e$. As $t$ varies, we obtain a map of vector bundles
$$
\Phi^{e,d} \colon T\times \C[x_1,\dots,x_r]^{\le e} \to E^{\le d}.
$$
Fix an arbitrary $e \ge 1$.
Since the map $\Phi^{e,Me}_0$ is injective, we have that $\Phi^{e,Me}_t$ is injective
for every $t$ in an open neighborhood of $0$.
Similarly, the map $\Phi^{e,me}_0$ is surjective,
and thus $\Phi^{e,me}_t$ is surjective for $t$ in an open neighborhood of $0$.
We conclude that $\Phi_t$
is a bijection for infinitely many values of $t \in T$ (in fact, for $t$ very general).

Therefore $R_t$ is isomorphic to a polynomial ring for infinitely many $t \in T$,
and we obtain an isomorphism of rings (depending on the choice of the
extensions $x_{t,i}$ of the sections $x_{0,i}$) between $R_t$ and $R_0$.
In particular, if $L$ is a relatively ample divisor on $X$,
then this isomorphism maps
the subring $\bigoplus_{m \ge 0}H^0(\O_{X_t}(mL))$ of $R_t$
to the subring $\bigoplus_{m \ge 0}H^0(\O_{X_0}(mL))$ of $R_0$.
This establishes an isomorphism between such subrings,
and therefore we obtain an isomorphism
$$
X_t = \Proj\Big(\bigoplus_{m \ge 0}H^0(\O_{X_t}(mL))\Big)
\cong \Proj\Big(\bigoplus_{m \ge 0}H^0(\O_{X_0}(mL))\Big) = X_0
$$
for infinitely many values of $t \in T$,
which shows that the deformation is locally trivial.
\end{proof}

\begin{rmk}\label{rmk:F_2-toric}
Using the degenerations of $\P^1\times\P^1$ into a quadric cone or $\F_2$ discussed in
Remark~\ref{rmk:F_2}, one sees immediately that the theorem
fails for smooth toric varieties that are just weak Fano or log Fano,
or if the singularities are canonical but not terminal.
The degenerations studied in \cite{Bat04} show that the theorem
fails for toric Fano varieties with terminal singularities that
are not simplicial (that is, not $\Q$-factorial).
\end{rmk}

\begin{rmk}
The arguments of the proof of Theorem~\ref{thm:rigidity-of-toric} also show that if
$f \colon X \to T$ is a projective flat deformation
of a Fano variety $X_0$ with $\Q$-factorial terminal singularities,
then the total coordinate ring $R(t)$ of a nearby fiber $X_t$
is a flat deformation of the total coordinate ring $R(0)$ of $X_0$.
\end{rmk}

\end{document}